\documentclass[11pt, a4paper,reqno]{amsart}

\usepackage{amsmath}
\usepackage{amsfonts}
\usepackage{amssymb}
\usepackage{amsthm}
\usepackage{enumerate}

\newtheorem{thm}{Theorem}[section]
\newtheorem{cor}[thm]{Corollary}
\newtheorem{lem}[thm]{Lemma}
\newtheorem{prop}[thm]{Proposition}

\theoremstyle{remark}
\newtheorem{rem}[thm]{Remark}
\numberwithin{equation}{section}

\theoremstyle{definition}

\newcommand{\R}{\mathbb R}

\newcommand{\C}{\mathcal C}
\newcommand{\B}{\mathcal B}

\newcommand{\dist}{\mathrm{dist}}

\def\XXint#1#2#3{{\setbox0=\hbox{$#1{#2#3}{\intr}$}
     \vcenter{\hbox{$#2#3$}}\kern-.5\wd0}}

\title{Extinction profile of complete non-compact solutions to the Yamabe flow}

\author{Panagiota Daskalopoulos}
\address{ {\bf P. Daskalopoulos:} Department of Mathematics, Columbia University, 2990 Broadway, New York, NY 10027, USA.}
\email{pdaskalo@math.columbia.edu}

\author{John King}
\address{ {\bf J. King:}  School of Mathematical Sciences, The University of Nottingham, University Park,
Nottingham, NG7 2RD}\email{john.king@nottingham.ac.uk
}

\author{Natasa Sesum}
\address{ {\bf N. Sesum:} Department of Mathematics, Rutgers University, 110 Frelinghuysen road, Piscataway,  NJ 08854, USA.}\email{natasas@math.rutgers.edu
}

\begin{document}

\maketitle

\begin{abstract}
This work addresses the  {\em singularity formation}  of  complete non-compact solutions to the conformally flat Yamabe flow whose conformal factors have {\em cylindrical behavior at infinity}. Their singularity profiles happen to be  {\em Yamabe solitons}, which are  {\em self-similar solutions}  to the fast diffusion equation satisfied by the conformal factor of the  evolving metric. 
The   self-similar profile is determined  by the second order asymptotics at infinity of the initial data
which is  matched  with  that of  the  corresponding self-similar solution. 
Solutions  may become  extinct at the extinction
time $T$ of the cylindrical tail  or may live longer than $T$. In the first case the singularity profile  is described by a {\em Yamabe shrinker}  that becomes extinct at time $T$. 
In the second  case, the singularity profile  is described by a {\em singular}  Yamabe shrinker slightly before $T$
 and by  a matching  {\em Yamabe expander}  slightly after $T$ . 

\end{abstract}

\section{Introduction}

We consider a complete non-compact  metric $g= u^{{4}/(N+2)}\, dx^2$ which is conformally equivalent to
the standard euclidean metric of  $\R^N$ and evolves by the {\em Yamabe flow} 
\begin{equation}
\label{eq-yf}
\frac{\partial g}{\partial t}  = - R \,  g 
\end{equation}
where $R$ denotes the scalar curvature with respect to  metric $g$. Our goal is to study the {\em singularity
formation}  of  metric $g$ at a singular time $T$,  under the assumption that the initial metric $g_0$ has  {\em cylindrical behavior} at infinity. 

This flow was introduced by R. Hamilton  \cite{H} as an approach to solve  the {\em Yamabe problem}  on manifolds of positive conformal Yamabe invariant.
It is the negative $L^2$-gradient flow of the total scalar curvature, restricted to a given conformal class. 
Hamilton \cite{H}  showed   the  existence of the  normalized Yamabe flow (which is the re-parametrization of (1.1) to keep the volume fixed) for all time; moreover   he established  the exponential convergence of the flow to a metric of constant scalar curvature under the assumption that the initial metric has negative  scalar curvature.

Since then,  there have been  a number of works on   the  convergence of the  Yamabe flow
on a compact manifold to a metric of constant scalar curvature.  Chow \cite{Ch} showed the  convergence of the flow  under the conditions that the initial metric is locally conformally flat and of positive Ricci curvature. 
The convergence of the flow for any locally conformally flat initial metric was shown by Ye  \cite{Ye}
(see also a relevant result of  Del Pino and Saez \cite{DS} for the conformally flat case).

Schwetlick and Struwe  \cite{SS}  obtained the convergence 
of the  Yamabe flow on a general compact manifold under 
a suitable Kazdan-Warner type of condition that rules out the formation of bubbles 
and that  is satisfied  (via the positive mass Theorem) in dimensions $3 \leq N \leq 5$. 
The convergence  result    for any general compact manifold  was established by
Brendle  \cite{S2} and \cite{S1}   (up to a technical assumption,  in dimensions $N \ge 6$, on the rate of vanishing of Weyl tensor at the points at which it vanishes):  
starting with any smooth  metric on a compact manifold, the normalized Yamabe flow   converges to a metric of constant scalar curvature.  

Even though the analogue of Perelman's monotonicity formula is still lacking for the Yamabe flow, one expects that  Yamabe soliton solutions  model finite time singularities. These are special solutions $g=g_{ij}$ of the Yamabe flow 
\eqref{eq-yf} for which there exist a 
{\em potential  function $P(x,t)$}  so that
$$(R-\rho)g_{ij} = \nabla_i\nabla_j P, \qquad \rho \in \{1,-1,0\} $$
where the covariant derivatives on the right hand side are taken with respect to metric $g(\cdot,t)$. 
Depending on the sign of the constant $\rho$,  a Yamabe soliton is  called a {\em Yamabe shrinker}, a {\em Yamabe expander} 
or a { Yamabe steady soliton}  if $\rho=1,-1$ or $0$ respectively. 
The classification of locally conformally flat Yamabe solitons
with positive  sectional curvature was recently established in \cite{DS2} (see also \cite{CSZ} and \cite{CMM}). 
It is shown in \cite{DS2} that such solitons are globally conformally equivalent to $\R^N$ and   correspond to self-similar solutions of the fast-diffusion equation \eqref{eqn-fd} satisfied by the conformal factor.  A complete description of those solutions is  given in \cite{DS2}. 
In \cite{CSZ} the assumption of positive sectional curvature was relaxed to that of nonnegative Ricci curvature.

Our goal in this work is to relate   the singularity profile of conformally flat solutions to the Yamabe flow
whose conformal factors have {\em cylindrical behavior at infinity} with a class of self-similar shrinking Yamabe solitons
that have   matched asymptotic  behavior at infinity. One special result in this direction was previously shown in  \cite{DS1} 
and \cite{BBDGV}, where the $L^1$ stability around the explicit Barenblatt profile was shown. 


\medskip

By observing that  the conformal metric $g= u^{{4}/(N+2)}\, dx^2$ has scalar curvature
$$R = - \frac{4(N-1)}{N-2}\, u^{-1} \, \Delta  u^{\frac {N-2}{N+2}}$$
it follows that the function $u$ evolves by the  fast diffusion equation $u_t = \frac{N-1}{m}\,\Delta u^m$, with exponent 
$m=(N-2)/(N+2)$.  Therefore studying the Yamabe flow equation \eqref{eq-yf} in the conformally flat case is equivalent to studying the fast diffusion equation on $\R^N$. It is well known \cite{HP}   that for any exponent  $0 < m < 1$   the Cauchy problem 
\begin{equation}
\begin{cases} \label{eqn-u}
u_t  = \Delta u^m  & \mbox{on} \,\,  \R^N \times (0,T)\\
u(\cdot,0) = u_0  & \mbox{on} \,\,  \R^N
\end{cases}
\end{equation}
with nonnegative and locally integrable initial data $u_0$ admits a unique weak solution and that bounded
solutions are  smooth. We refer the reader to \cite{DK} and \cite{CV} for extensions 
of the results in \cite{HP} to the case that
the initial data is a nonnengative Borel measure $\mu_0$ and to \cite{King} for formal results that suggest that, within the setting \eqref{eqn-u}, many of the phenomena described below are more generally relevant to the range $0 < m < \frac{N-2}{N}$ of exponents.

From now on we will  fix $m = (N-2)/(N+2)$ and  set
$$n:= 1-m = \frac 4{N+2}.$$ 
We will assume  that the initial metric 
$g_0 = u_0^{{4}/(N+2)}\, dx_idx_j$ is complete, non-compact and has  {\em cylindrical behavior} at infinity,
namely
\begin{equation} \label{eqn-u0infty}
u_0(x) = \left(\frac{C^* T}{|x|^2}\right)^{1/n} \big ( 1 + o(1) \big ), \qquad \mbox{as} \,\, |x| \to \infty
\end{equation} 
with $C^*$  given by
\begin{equation}
\label{eqn-cstar}
C^*:= \frac{2\, \big ( ((1-m) N -2\big )}n, \qquad n=1-m, \quad m=\frac{N-2}{N+2}
\end{equation}
and $T >0$ any positive constant.  
One observes that the function  
\begin{equation}
\label{eqn-cyl2}
 \C (x,t) = \left(\frac{C^*\, (T-t)}{|x|^2}\right)^{1/n}
\end{equation}
defines a  cylindrical solution of \eqref{eqn-u},   namely  $g(t) = \C^{{4}/(N+2)}(\cdot,t)\, dx^2$
represents a   shrinking cylindrical metric. Its  initial data $\C_0 := \C(\cdot,0)$ satisfies  \eqref{eqn-u0infty} and the solution becomes
extinct at time $t=T$.
This suggests that  the cylindrical tail of any solution to \eqref{eqn-u} that  satisfies
\eqref{eqn-u0infty} becomes extinct at time $T$. Indeed, it will be shown  in Proposition \ref{prop-tail}  that if the initial data $u_0(x)$ satisfies \eqref{eqn-u0infty} then for the solution $u$ we have
\begin{equation}
\label{eqn-tail0}
u(x,t) = \left(\frac{C^*(T-t)}{|x|^2}\right)^{1/n} (1+o(1)), \qquad \mbox{as} \,\, |x| \to \infty.
\end{equation}

We will see in this work that  the solution $u$ starting at $u_0$ that satisfies \eqref{eqn-u0infty}  {\em may or  may not become extinct at time $T$},   depending  on
the {\em second order asymptotic behavior}, as $|x| \to \infty$, of the cylindrical tail of the initial data. 
In either case the 
metric $g(t)=u^{{4}/(N+2)}(\cdot,t)\, dx^2$ will develop a singularity at time $T$. Our goal  is to study 
these singularities. 
We will show in  sections \ref{sec-conv} and \ref{sec-longer} that  rescaled limits of   solutions  $u$ with initial condition satisfying \eqref{eqn-u0infty} 
behave near a singularity at time $T$ as self-similar shrinking solutions (Yamabe shrinkers).  These are  special solutions of  the fast-diffusion  equation
\begin{equation} \label{eqn-fd}
u_t =\Delta u^{\frac{N-2}{N+2}}
\end{equation} 
of  the form
\begin{equation}
\label{eqn-ss}
U(x,t) = (T-t)^\alpha \, f(y), \quad y= x \, (T-t)^\beta,  \qquad \alpha = \frac{1+ 2\beta}n,  \,\,\,  \beta >0.
\end{equation}
It follows that the  function   $f$  satisfies  the elliptic equation 
\begin{equation}\label{eqn-ell}
 \Delta f^{\frac {N-2}{N+2}}+ \beta  \, y \cdot \nabla f    +  \alpha \, f =0. 
\qquad \mbox{on}\,\,\, \R^N
\end{equation}
It is well known (in \cite{V}, Section 3.2.2)  that, for any given $\beta >0$ and  $\alpha = (1+2\beta)/n$,   equation \eqref{eqn-ell} admits
an one parameter family $f_\lambda$, $\lambda >0$,  of radially symmetric smooth positive solutions that have  cylindrical behavior at infinity, namely
\begin{equation}
\label{eqn-binfty}
f_\lambda(y) =  \left ( \frac{C^*}{|y|^2}  \right )^{1/n}(1+o_{\lambda}(1)), \qquad \mbox{as} \,\, y \to \infty.
\end{equation}
with $C^*$ given by \eqref{eqn-cstar}.   We will refer to them as  to {\em cigar solitons}. 
The parameter $\lambda$ is just a dilation parameter. Indeed, it follows from the results in \cite{DS2} that smooth solutions of equation \eqref{eqn-ell}
are radially symmetric and they  are uniquely determined by their value at the origin. 
 In the special case that   $\alpha  = \beta  N$ the solutions are given in the closed form
 \begin{equation}\label{eqn-ub}
\B_\lambda (y) = \left ( \frac {C^* }{\lambda^2 + |y|^2} \right )^{1/n}
\end{equation}
and we will refer to them as  Barenblatt profiles.

In order to study the singularities  of  a  metric $g=u^{4/(N+2)} dx^2$ evolving by \eqref{eqn-u}  and with initial data satisfying \eqref{eqn-u0infty}   we need to understand  the second order asymptotic behavior 
at infinity of the self-similar profiles $f_\lambda$.  
 We will  achieve this  in section \ref{sec-asymp} by 
linearizing  equation \eqref{eqn-ell} around the cylindrical solution. It will be more convenient to
work in  cylindrical coordinates where the cylindrical solution becomes constant. 
Let $\gamma_{1,2}$ be the solutions to  the characteristic equation of the corresponding linearized equation
(that is  equation \eqref{eqn-w9} in section \ref{sec-asymp}). They satisfy 
\begin{equation}
\label{eqn-eqgamma}
\gamma^2 + \beta(N-2)\gamma + (N-2) = 0,
\end{equation}
which gives 
\begin{equation}
\label{eqn-gamma}
\gamma_{1,2} = \frac{\beta(N-2) \mp \sqrt{\beta^2(N-2)^2 - 4(N-2)}}{2}.
\end{equation}
We see that we need to have $\beta \geq {2}/{\sqrt{N-2}}$ in order for $\gamma_{1,2}$ to be real
and the corresponding  solution to have non-oscillatory  behavior. 

Our first result concerns   the second order asymptotics of smooth profiles $f$ on $\R^N$ which appear to model the singular behavior of some evolving metrics $g=u^{4/(N+2)} dx^2$ that  become extinct at a singular  time $T$.

\begin{thm}
\label{thm-ss}
Let $m = (N-2)/(N+2)$, $n=1-m$, $N \ge 3$, $C^* = 2\, ( (1-m) N - 2 )/n$, $\beta_0 := {2}/{\sqrt{N-2}}$ and 
$\beta_1 := 1/(2m)$. The following hold:
\begin{itemize}
\item
Let  $N \ge 6$ and $\beta > \beta_0$ or $3 \le N < 6$ and $\beta > \beta_1$: For any $B >0$ there exists a unique 
radially symmetric smooth solution  $f_{\beta,B}$ of \eqref{eqn-ell} that  satisfies 
\begin{equation}\label{eqn-asympt1}
f_{\beta,B}(y) = \left (\frac{C^*}{|y|^2} \right )^{1/n}\,  \big (1 - {B}\, {|y|^{-\gamma}} + o_B(|y|^{-\gamma})  \big)
\end{equation}
with $\gamma=\gamma_1$ given by \eqref{eqn-gamma}.
\item
Let  $3 \leq  N < 6$ and $\beta_0 < \beta < \beta_1$:  For any $B < 0$ there exists a unique radially symmetric smooth
solution  $f_{\beta,B}$ of \eqref{eqn-ell} that   
satisfies \eqref{eqn-asympt1} with 
$\gamma=\gamma_1$ given by \eqref{eqn-gamma}. 
\item Let   $3 \leq  N < 6$ and $\beta = \beta_1$:   For any $B < 0$ there exists a unique radially symmetric smooth
solution  $f_{\beta,B}$ of \eqref{eqn-ell} that 
satisfies \eqref{eqn-asympt1} with 
$\gamma=\gamma_2=2$ and which is given in closed form by \eqref{eqn-ub}. 
\end{itemize}
In all of the above cases we will denote by $U_{\beta,B}$ the self-similar solution of equation \eqref{eqn-fd}.
It is given in terms of $f_{\beta,B}$ by \eqref{eqn-ss} where $f_{\beta,B}$ solves \eqref{eqn-ell}.
\end{thm}


While the previous theorem provides a complete description of 
smooth self-similar solutions of equation \eqref{eqn-fd} of the form \eqref{eqn-ss} with cylindrical behavior 
at infinity,  one 
 may ask whether there exist other radially symmetric  self-similar solutions  with singular behavior at $r=0$. The answer to this question is indeed affirmative as stated in our next result. We will  see in
section \ref{sec-longer} that such solutions 
model the  behavior of evolving metrics $g=u^{4/(N+2)} dx_idx_j$ that  do not become extinct at a singular  time $T$, but instead pinch off.

\begin{thm}
\label{thm-ss-sing}
Let $N \ge 3$, $m = (N-2)/(N+2)$, $n = 1 - m$, $C^* = 2\, ( (1-m) N - 2 )/n$ and $\beta_1 = 1/(2m)$. Then for any $\beta > \beta_1$ and $B >0$, there exists 
a unique radially symmetric solution $g_{\beta,B}$ of equation \eqref{eqn-ell} that  is smooth on $\R^N \setminus \{0\}$
and satisfies 
\begin{equation}
\label{eqn-asympt2}
g_{\beta,B}(y) =\left (\frac{C^*}{|y|^2} \right )^{1/n}\,  (1 + B|y|^{-\gamma} + o_B(|y|^{-\gamma})), \qquad \mbox{as} \,\,\, |y| \to \infty 
\end{equation}
and 
\begin{equation}\label{eqn-orig} 
g_{\beta,B}(y) = K_B \,  |y|^{-\alpha /\beta} (1+o(1)), \qquad \mbox{as} \,\,\, |y|\to  0
\end{equation}
with $K_B$  a constant depending on $B$ and $\gamma := \gamma_1$. We will denote by 
$V_{\beta,B}$  the self-similar solution of equation \eqref{eqn-fd}
which is given in terms of $g_{\beta,B}$ by \eqref{eqn-ss}.
\end{thm}

One easily concludes, using the behavior of $g_{\beta,B}$ at the origin,  that
\begin{equation}
\label{eq-V-lim}
\lim_{ t \to T^-} V_{\beta,B}(x,t) = K_B \,  |x|^{-\alpha /\beta} \qquad \forall x \neq 0.
\end{equation}
For any $T>0$ and any $K >0$  we will denote by 
\begin{equation}\label{eqn-expan}
W_{\beta,K}(x,t) = (t-T)^{\alpha} h_{\beta,K}(x\, (t-T)^{\beta}), \qquad t> T
\end{equation}
the forward self-similar
solutions (Yamabe expanders) that satisfy 
\begin{equation}\label{eqn-hhh1}
h_{\beta,K}(y) = K \, |y|^{-\alpha/\beta}(1+o(1)), \qquad  \mbox{as} \,\, |y|\to 0 
\end{equation}
and
\begin{equation}\label{eqn-hhh2}  
h_{\beta,K}(y) = D_K |y|^{-(N+2)} (1+o(1)), \qquad  \mbox{as} \,\, |y|\to+\infty 
\end{equation}
with $D_K$ a constant depending on $K$. In \cite{V} Vazquez proves the existence of those solutions starting at   $W_{\beta, K}(x,T) = K\,  |x|^{-\alpha/\beta}$.  The existence of such solutions and their intermediate asymptotic role was conjectured in \cite{King} 
on the basis of a phase-plane analysis.

\medskip
We will see in sections \ref{sec-conv} and \ref{sec-longer}  that    the singularity profile of the metric $g=u^{4/(N+2)} dx_idx_j$  evolving by \eqref{eqn-u} at a singular time $T$ is closely related to the self-similar solutions
given above. In particular, the smooth self-similar solutions $U_{\beta,B}$ model the singular behavior  of some solutions $u$ in the case that $u(\cdot,T)$ vanishes identically  at time $T$, while the singular solutions 
$V_{\beta,B}$ and $W_{\beta,K}$ model the singularity of some solutions $u$ in the case
that $u(\cdot,T)$ does not vanish  identically  at the extinction time $T$ of the cylindrical tail. 

In describing the asymptotic profile of the solution slightly before time $T$ we will  consider the rescaling from the left defined by
\begin{equation}
\label{eq-resc-left}
\bar{u}(y,\tau) := (T-t)^{-\alpha} u(y\, (T-t)^{-\beta},t) |_{t= T(1-e^{-\tau})}, \qquad 
(y,\tau)\in \mathbb{R}^N \times (0,\infty).
\end{equation}
In describing the asymptotic profile of the solution slightly after time $T$ (if the solution lives for $t\in [0,T^*)$ and $T^* > T$) we will consider the rescaling from the right defined by
\begin{equation}
\label{eq-resc-right}
\hat{u}(y,\tau) := (t-T)^{\alpha} u(y\,(t-T)^{\beta},t)|_{t= T(1 + e^{\tau})}, \qquad (y,\tau) \in \mathbb{R}^N
 \times (-\infty, \tau^* )
\end{equation}
with $\tau^*$ such that $T^*= T(1 + e^{\tau^*})$. 
It follows by direct computation that both $\bar{u}$ and $\hat{u}$  satisfy the nonlinear Fokker-Plank  type equation
\begin{equation}
\label{eq-resc100}
\bar{u}_{\tau} = \Delta \bar{u}^m + \beta \, \mbox{div}  (y \cdot \bar{u}) + (\alpha - \beta N)\,  \bar{u}.
\end{equation}

\medskip
Let us begin by discussing the case when the solution with the cylindrical behavior at infinity becomes extinct at the time $T$ when its cylindrical tail disappears.   We will  assume in this case that either
\begin{itemize}
\item
$N \geq 3$ and $\beta \ge \beta_1$ (or equivalently   $N  \beta \geq  \alpha$), or 
\item 
$N \geq 6$ and $\beta_0 < \beta < \beta_1$.  
\end{itemize}
The condition $\beta \ge  \beta_0:=2/\sqrt{N-2}$ is imposed so that the self similar solution $U_{\beta,B}$ has non-oscillating
behavior as $ |x|\to +\infty$.  The common feature in both considered cases is that the difference of two self-similar solutions 
$$|U_{\beta,B_1} - U_{\beta,B_2}| \notin L^1(\mathbb{R}^N), \qquad \mbox{if}\,\,\, B_1 \neq B_2.$$

The next two theorems  generalize the result  proved in \cite{DS1} in the special case when $\beta = \beta_1$ (see  also in \cite{BBDGV}
for an improvement of the result in \cite{DS1} shown independently). 
Our first result is concerned with the case  $\beta \ge \beta_1$ in all dimensions  $N \ge 3$. 
\begin{thm}
\label{thm-conv}
Let $\beta \ge \beta_1$ and let $u: \mathbb{R}^N\times[0,T) \to \mathbb{R}$ be a solution to
\eqref{eqn-u} with the initial data $u_0$ satisfying $0 \le u_0 \le U_{\beta, B_1}(\cdot,0)$, for some $B_1 >0$. 
 Assume in addition that 
\begin{equation}
\label{eq-int}
u_0 -  U_{\beta,B} \in L^1(\mathbb{R}^N)
\end{equation}
for some $B > 0$. Then,  the rescaled solution  $\bar{u}$ given by \eqref{eq-resc-left} converges as $\tau\to \infty$  uniformly  on compact subsets of $\R^N$ to the self-similar solution $U_{\beta,B}$. Moreover, we also have  convergence in the $L^1(\mathbb{R}^N)$ norm. If $\beta > \beta_1$ the convergence is exponential.
\end{thm}

In the case when $\beta < \beta_1$ we will restrict ourselves to $N \ge 6$. 
Let  $\bar \C(x) = \left(C^*/|x|^2\right)^{1/n}$ with  $C^* = 2\, ( (1-m) N - 2 )/n$ denote the rescaled cylinder
which is a  singular solution  to
$$\Delta \bar u^m + \beta\, \mbox{div}\, (x\cdot\nabla \bar{u}) = 0.$$
We define the weighted $L^1$-space with weight $\bar{C}^{p_0}$ for some $p_0 \in (0, 2m)$ as
\begin{equation}\label{eqn-wl1}
L^1(\bar{\C}^{p_0}, \mathbb{R}^N) := \{f\, \, |\, \int_{\mathbb{R}^N} |f(x)| \, \bar{\C}^{p_0}(x)\, dx < \infty\, \}.
\end{equation}
Note  that $\bar{\C}^{p_0}$ is integrable around the origin for any $p_0\in (0,2m)$.
We have the following result.

\begin{thm}
\label{thm-conv2}
Let $\beta_0 < \beta < \beta_1$ with $N \ge 6$ and let $u:\mathbb{R}^N\times[0,T)\to\mathbb{R}$ be a solution to \eqref{eqn-u} with the initial data $u_0$ satisfying  $0 \le u_0 \le U_{\beta, B_1}(\cdot,0)$, for some $B_1 >0$.  Assume in addition that
\begin{equation}
\label{eq-int2}
u_0 - U_{\beta,B} \in L^1(\bar{\C}^{p_0}, \mathbb{R}^N)
\end{equation}
for some $B > 0$,  where $p_0 := m\left(1-\beta+\sqrt{\beta^2 - \frac{4}{N-2}}\right)$.
Then the rescaled function $\bar{u}$ given by \eqref{eq-resc-left} converges as $\tau\to\infty$ uniformly  on compact subsets of $\mathbb{R}^N$ to the self-similar solution $\bar{U}_{\beta,B}$. 
Moreover, we also have  convergence in the weighted $L^1(\bar{\C}^{p_0},\mathbb{R}^N)$.
\end{thm}

\begin{rem}
Note that when $\beta < \beta_1$ the implication of the $L^1$ contraction principle under rescaling \eqref{eq-resc-left} is inconclusive. The choice of $p_0$ as in Theorem \ref{thm-conv2} will allow us to replace the usual $L^1$ contraction principle with the weighted $L^1$ contraction principle with the weight being $\bar{\C}^{p_0}$. Note also that $|U_{\beta,B_1} - U_{\beta,B_2}| \notin L^1(\bar{\C}^{p_0}, \mathbb{R}^N)$  if $B_1 \neq B_2$. This would not be true for any weight $\bar{\C}^q$ for $q > p_0$.
\end{rem}

In section \ref{sec-longer} we will discuss the singular behavior of 
solutions $g=u^{4/(N+2)} dx_idx_j$ to \eqref{eq-yf} with  cylindrical
behavior at infinity that live past the extinction  time $T$ of the cylindrical tail  and become compact at time $T$.  We will assume that the initial data $u_0 \in L^{\infty}_{{loc}}(\R^N)$ and satisfies 
\begin{equation}\label{eqn-L1}
u_0(x) - V_{\beta,B}(x) \in L^1(\mathbb{R}^N),
\end{equation}
where $V_{\beta,B}(x,t)=(T-t)^\alpha \, g_{\beta,B}(x(T-t)^\beta)$ is one of the {\em singular} at the origin  self similar solutions  given by Theorem \ref{thm-ss-sing} with $\beta > \beta_1$  and some $B >0$. In addition we will assume that $u_0$ satisfies the asymptotic behavior
\begin{equation}\label{eqn-u0}
u_0(x) = \left(\frac{C^* T}{|x|^2}\right)^{1/n} (1 + B\, |x|^{-\gamma} + o(|x|^{-\gamma}), \qquad \mbox{as}\,\,  |x| \to \infty, 
\end{equation}
for some $B>0$ and $\gamma := \gamma_1$.
We will see   that  condition \eqref{eqn-L1}  implies that  the solution to \eqref{eqn-u} with initial data $u_0$
is strictly positive at the exitinction time $T$ of the cylindrical tail. 
Denote  by $T^* >T$ the extinction time of the solution $u$. We  have  the following result. 

\begin{thm} 
\label{thm-longer}
Let $\beta \ge \beta_1$ and let   $u : \mathbb{R}^N \times [0,T^*) \to \mathbb{R}$ be the solution to \eqref{eqn-u} with the initial data $u_0   \in L^{\infty}_{{loc}}(\R^N)$ satisfying \eqref{eqn-L1} and \eqref{eqn-u0}.
 Then, the following hold
\begin{itemize}
\item The solution $u$ is non-zero at time $T$, i.e. $u(\cdot, T) >0$. 
\item The cylindrical tail
becomes extinct at time $T$ according to \eqref{eqn-tail0} and the 
rescaled solution  $\bar{u}(\eta,\tau)$ given by \eqref{eq-resc-left} converges   as 
$\tau\to \infty$   uniformly  on compact subsets of $\R^N$ 
to the self-similar profile $g_{\beta,B}(\eta)$ that  satisfies \eqref{eqn-asympt2} and \eqref{eqn-orig}.  
\item The rescaled solution $\hat{u}$ given by \eqref{eq-resc-right} 
converges  as $\tau \to -\infty$ uniformly on compact subsets of $\R^N \setminus \{0\}$   to the self-similar profile $h_{\beta,K}$ that  satisfies \eqref{eqn-hhh1} and \eqref{eqn-hhh2}
with the same constant $K=K_B$ as in \eqref{eqn-orig}.  
\item For $t > T$, the solution $u$ satisfies the bound 
$u(x,t) = O(|x|^{-(N+2)})$, as $|x| \to \infty$.  
\end{itemize} 
\end{thm}

\begin{rem}
Let  $\beta \ge \beta_1$ and $\gamma := \gamma_1$. 
In section \ref{sec-longer} we will  also see that there exist  solutions $g = u^{4/(N+2)} dx_i dx_j$ to \eqref{eq-yf} with initial data satisfying $u_0 - U_{\beta,B} \in L^1(\R^N)$ and 
\begin{equation}
\label{eqn-u000}
u_0(x) = \left(\frac{C^* T}{|x|^2}\right)^{1/n}(1 - B\, |x|^{-\gamma} + o(|x|^{-\gamma}), \qquad
\mbox{as} \,\, |x| \to \infty
\end{equation}
with $B > 0$  that  {\em live longer}  than the vanishing time $T$ of their cylindrical tail. The rescaling $\bar u$ of $u$ given  by \eqref{eq-resc-left} will still converge to $f_B$, however, the convergence  will only  be
uniform  on $\R^N \setminus \{0\}$, reflecting  the  non-vanishing of the solution at time $T$. 
 This in particular shows that the upper bound $u_0 \leq U_{\beta,B_1}$ in Theorem \ref{thm-conv}  is necessary. 
\end{rem}

{\noindent \em Further Discussion.}  It  would be nice to   understand better the singularity formation of the  Yamabe flow on 
complete  non-compact manifolds.  
One of the ultimate goals would be to show that the singularity of such a flow  is modeled by one of 
the gradient Yamabe solitons. For a general statement like that some sort of   monotonicity formula would play an important role. By the results in \cite{CSZ, DS2} the gradient Yamabe solitons with nonnegative Ricci curvature  are  well understood and  have been shown to be globally conformally equivalent to $\R^N$. 
However, the class of solutions discussed in Theorem \ref{thm-longer} provides prototypes of
singularities that are not  globally conformally flat. By the results in  \cite{CSZ, DS2} the {\em Ricci curvature}  
of those solutions must  {\em change  sign}.  A characterization of all gradient Yamabe solitons
is then necessary.

%

\medskip
The {\em organization}  of the paper is as follows. In section \ref{sec-ss} we discuss the existence of smooth self-similar solutions  $U_{\beta,B}$ and singular self-similar solutions $V_{\beta,B}$ and $W_{\beta,K}$. In section \ref{sec-asymp} we prove Theorem \ref{thm-ss} and Theorem \ref{thm-ss-sing}. In section \ref{cyl-beh} we prove Proposition \ref{prop-tail}, which claims the cylindrical tail in a solution persists up to the vanishing time of the cylinder. The proofs of Theorems \ref{thm-conv} and  \ref{thm-conv2} are given in section \ref{sec-conv}. These theorems discuss the asymptotic profile of solutions that become extinct at the time that their cylindrical tail disappears. In section \ref{sec-longer} we discuss solutions that live longer than the time of disappearance of their cylindrical tail and we show the precise singularity profile of those solutions, as stated in Theorem \ref{thm-longer}.

\centerline{\bf Acknowledgements}
P. Daskalopoulos  has been partially supported
by NSF grant 0604657.
N. Sesum has been partially supported by NSF grants 0905749 and 1056387.

\section{Self-similar solutions}
\label{sec-ss}

Consider self-similar solutions $U(x,t)$ of fast diffusion equation \eqref{eqn-u} in dimensions $N \ge 3$ of the form
$$U(x,t) = (T-t)^{\alpha} f(y), \qquad y = x(T-t)^{\beta}, \qquad \alpha = \frac{1+2\beta}{n}, \,\,\, \beta > 0$$
where $f(y)$ is a radial solution of the elliptic equation \eqref{eqn-ell}. We recall that we have set $n=1-m$, $m=(N-2)/(N+2)$ and that $N \geq 3$. 

In this section we will  discuss the existence and geometric properties of  three different kinds of self-similar solutions 
(Yamabe solitions) that  will be used in further singularity analysis. 
In the next  section we will discuss their second order asymptotic behavior as $|y| \to \infty$,  which is needed to understand the singular profiles of solutions to \eqref{eqn-u} with cylindrical behavior at infinity.  In what follows let $\alpha := (2\beta+1)/{n}$ and $\beta > 0$.

\begin{enumerate}
\item[(i)]
We  denote by $U_{\beta,B}(x,t) = (T-t)^{\alpha} f_{\beta,B}(x(T-t)^{\beta})$,  $t \in (-\infty,T)$ and $B>0$ a two parameter family of radially symmetric smooth self-similar solutions  satisfying the  cylindrical behavior \eqref{eqn-binfty} at infinity. 
Their  existence  for any $\beta > 0$ is well known \cite{V}.
\item[(ii)]
We  denote by   $V_{\beta,B}(x,t) =  (T-t)^{\alpha} g_{\beta,B}(x(T-t)^{\beta})$,  $t\in (-\infty,T)$ and $B>0$  a two parameter family of radially symmetric singular at the origin self-similar solutions with the cylindrical behavior 
\eqref{eqn-binfty} at infinity.  The behavior of the   profile function $g_{\beta,B}$ at the origin is given by  \eqref{eqn-orig},  where $K_B$ is a constant depending on $B$. The existence of these  solutions will be proved in Proposition \ref{lem-V-B} below.
\item[(iii)]
We  denote by $W_{\beta,K}(x,t) = (t-T)^{\alpha} h_{\beta,K}(x(t-T)^{\beta})$,  $t\in (T,\infty)$ and $K > 0$, a two parameter family of radially symmetric forward self-similar solutions with  profile function $h_{\beta,K}$ satisfying \eqref{eqn-hhh1} and \eqref{eqn-hhh2} 
with $D_K$ a constant depending on $K$. In \cite{V} Vazquez proved the existence of these solutions starting with the initial data $W_{\beta,K}(x,T) = K\,  |x|^{-\alpha/\beta}$.
\end{enumerate}

\begin{rem}
In (i) and (ii) above we parametrize the self-similar profiles $f_{\beta,B}$ and $g_{\beta,B}$ by the
constant $B$ that  appears in second-order asymptotics of the  corresponding self-similar solutions (see \eqref{eqn-asympt1} and \eqref{eqn-asympt2}).
\end{rem}

\subsection{Geometric properties of Yamabe solitons}
We summarize below some of the geometric properties of the Yamabe solitons that were introduced above. 
\begin{itemize}
\item The Yamabe soliton defined by $g = U_{\beta,B}^{\frac{4}{N+2}}\, dx^2$, where $U_{\beta,B}$ is described in (i) above is a  complete conformally flat  radially symmetric Yamabe shrinker  on $\mathbb{R}^N$ satisfying equation $(R-1) g_{ij} = \nabla_i\nabla_j P_u$ for  a radially symmetric potential function $P_u$. This soliton behaves as a cylinder at infinity. In \cite{DS2} we showed that for $\beta \ge \beta_1$ they have positive sectional curvature. 

\item The Yamabe soliton defined by $g = V_{\beta,B}^{\frac{4}{N+2}}\, dx^2$, where $V_{\beta,B}$ is described in (ii) above is a complete locally conformally flat radially symmetric Yamabe shrinker on $\mathbb{R}^N\backslash \{0\}$ satisfying equation $(R-1)g_{ij} = \nabla_i\nabla_j P_v$  for a radially symmetric function $P_v$. This soliton also behaves as a cylinder at infinity. It is singular at the origin and therefore by the classification and rigidity result in \cite{CH} it has to have somewhere negative Ricci curvature (since it is not flat,  not locally isometric to a cylinder,  not globally conformally flat and not conformal to a spherical spaceform; the last is true because our soliton is not compact). It is easy to check the completeness of our solution at the origin where $V_{\beta,B} \sim |x|^{-\alpha/\beta}$ as $|x| \to 0$. The completeness follows from
$$\dist_g(x,0) \ge C \int_0^x |y|^{-\frac{2\beta+1}{2\beta}} \, dy = +\infty$$
since $\beta > 0$, implying that  the distance to the origin is infinity. This soliton has 2 ends.

\item The Yamabe soliton defined by $g = W_{\beta,K}^{\frac{4}{N+2}}\, dx^2$, where $W_{\beta,K}$ is described in (iii) above is a complete locally conformally flat radially symmetric Yamabe expander on $\mathbb{R}^N\backslash \{0\}$ satisfying equation $(R+1)g_{ij} = \nabla_i\nabla_j P_w$  for a radially symmetric potential function $P_w$. This soliton admits  the spherical behavior at infinity, which means it is compact on one end. It behaves at the origin like the previously discussed Yamabe shrinker on $\mathbb{R}^N\backslash\{0\}$, which means it is complete at the origin (one can also check that the area around the origin is infinite). This soliton metric has only one end and by the same arguments as for the previously discussed Yamabe shrinker has somewhere negative Ricci curvature. 
\end{itemize}
 

\subsection{Scaling and  Monotonicity of self-similar solutions}\label{sec-scal}
In this section we will use the  maximum principle and the scaling properties  of equation \eqref{eqn-ell} to establish the monotonicity of the 
self-similar solutions $U_{\beta,B}$ and $V_{\beta,B}$ with respect to the parameter $B$ 
and with respect to the radial variable $\eta=|y|$.

We begin by noting   the relation between  the parameter $B$ in the self-similar profiles 
$f_{\beta,B}$ and $g_{\beta,B}$ (that  satisfy  \eqref{eqn-asympt1} and \eqref{eqn-asympt2} respectively) 
and the behavior of those profiles at the origin. To this end, let 
$$U_{\beta, B_0}(y,t) = (T-t)^\alpha f_{\beta,B_0}(x\, (T-t)^\beta) \quad \mbox{and} \quad V_{\beta, \bar B_0} (y,t) = (T-t)^\alpha g_{\beta,\bar B_0} (x\, (T-t)^\beta)$$
denote the self-similar solutions with  
$$f_{\beta, B_0} (0) = 1  \qquad \mbox{and} \qquad \lim_{|y| \to 0} |y|^{\alpha/\beta} g_{\beta, \bar B_0}(y) =1$$
respectively. 
Clearly,  $B_0=B_0(\beta,N) >0$ and $\bar B_0=\bar B_0(\beta,N) >0$.
Both profiles  satisfy \eqref{eqn-asympt1} and \eqref{eqn-asympt2} respectively, which in particular imply  that 
\begin{equation}\label{eqn-cyl3}
\lim_{|y| \to +\infty} |y|^{2/n} f_{\beta,B_0}(y)= \lim_{|y| \to +\infty} |y|^{2/n} g_{\beta,\bar B_0}(y)=C^*.
\end{equation} 
The  rescaled solutions of equation \eqref{eqn-ell} that preserve \eqref{eqn-cyl3}  are given by 
$$f^\lambda_{\beta,B_0}  (y) := \lambda^{2/n} \, f_{\beta,B_0}(\lambda y)  \qquad \mbox{and} \qquad g^\lambda_{\beta,\bar B_0} (y) := 
\lambda^{2/n} \, g_{\beta,\bar  B_0}(\lambda y) $$
and satisfy 
\begin{equation}\label{eqn-zero}
f^\lambda_{\beta,B_0}  (0) = \lambda^{2/n}   \qquad \mbox{and} 
\qquad \lim_{|y| \to 0} |y|^{\alpha/\beta} g^\lambda_{\beta,\bar B_0}(y)  =\lambda^{-1/(n\beta)}. 
\end{equation}
It follows that $f^\lambda_{\beta,B_0} = f_{\beta,B_\lambda}$ and 
$g^\lambda_{\beta,\bar B_0} = g_{\beta,\bar B_\lambda}$ where $ f_{\beta,B_\lambda}$ and 
$ g_{\beta,B_\lambda}$ satisfy \eqref{eqn-asympt1} and \eqref{eqn-asympt2} respectively
with 
$B_\lambda = B_0 \lambda^{-\gamma}$ and $\bar B_\lambda = \bar B_0 \lambda^{-\gamma}$. 
To simplify the notation in what follows we set $f_\lambda (\eta):=f^\lambda_{\beta,B_0}(y)$ and 
$g_\lambda(\eta):=g^\lambda_{\beta, \bar B_0}(y)$, $\eta=|y|$. 

\begin{lem} Assume that $N\beta \geq  \alpha$ (or  equivalently,   $\beta \geq \beta_1$). If $0 < \lambda_1 < \lambda_2$, then  
\begin{equation}\label{eqn-order}
f_{\lambda_1} < f_{\lambda_2}  \qquad \mbox{and} 
\qquad g_{\lambda_1} > g_{\lambda_2}.
\end{equation}
\end{lem}

\begin{proof} We observe that in the case   $N\beta \geq  \alpha$  radially symmetric weak solutions of equation \eqref{eqn-ell} cannot cross: if they coincide at a point $\eta_0=|y_0|$ they must be the same. 
This follows by the simple observation that  uniqueness of weak solutions of equation  \eqref{eqn-ell}
holds on $B_{r_0}(0)$ when $N\beta > \alpha$. Assuming that  $\lambda_1 < \lambda_2$,  it then  follows from 
\eqref{eqn-zero} that \eqref{eqn-order} hold. 
\end{proof}

As a consequence of the previous lemma, we have the following.

\begin{prop} 
\label{cor-mon}
Assume that $N\beta \geq  \alpha$ (or  equivalently,   $\beta \geq \beta_1$). For any $0 < B_1 < B_2$ we have
$$  U_{\beta, B_2} < U_{\beta,B_1} < \C < V_{\beta,B_1} < V_{\beta, B_2}$$
with $\C$ denoting the cylindrical solution given by \eqref{eqn-cyl2}. 
In addition, if $\eta = |y|$ 
\begin{equation}\label{eqn-moncyl}
\frac{d}{d\eta} \big ( \eta^{2/n} f_{\beta,B} (\eta) \big ) >0   \qquad \mbox{and} 
\qquad \frac{d}{d\eta} \big ( \eta^{2/n} g_{\beta,B}(\eta) \big ) <0
\end{equation}
holds for any $B >0$.
\end{prop}

\begin{proof} Similarly to the proof of the previous lemma,  any two solutions   from the  $f_{\beta, B_i}$ and $g_{\beta,B_i}$, $i=1,2$, 
of \eqref{eqn-ell} cannot cross 
each other when $N\beta >\alpha$. Hence, for any $0 < B_1 < B_2$ the monotonicity 
$U_{\beta, B_2} < U_{\beta,B_1} <  V_{\beta,B_1} < V_{\beta, B_2}$
readily follows from the behavior of those solutions at infinity. 
In addition, the monotonicity of the profiles $f_\lambda$ and $g_\lambda$, noted in the previous lemma, readily implies that
\eqref{eqn-moncyl} holds. This in particular implies that, for any $B_1, B_2 >0$, we have
$ U_{\beta,B_1} < \C <  V_{\beta, B_2}$ and the proof of the proposition is complete. 

\end{proof}

\section{Precise asymptotics of self-similar solutions}
\label{sec-asymp}

We will establish in this section   the precise asymptotics,  up to second order,  of the smooth self-similar solutions $U_{\beta,B}$ and the singular self-similar solutions $V_{\beta,B}$,  both of which have been discussed in the previous section. 
More precisely,  we will  prove Theorems \ref{thm-ss} and  \ref{thm-ss-sing}. 

It will be convenient to work in cylindrical coordinates.  Consider for the moment   any   solution $u$ of  fast-diffusion equation \eqref{eqn-fd}
that is defined on  $\R^N\times (-\infty , T)$, $T >0$ and vanishes at time $T$. Assuming that    $u(r,t)$ is radial,  set 
\begin{equation}\label{eqn-cylcor}
v(s,\tau )   = (T-t)^{-1/(p-1)} r^{2/(p-1)} \, u^m(r,t), \qquad r = e^s, \,\,\, \tau = - \log \, (T-t),
\end{equation}
where we recall that $p:=1/m=(N+2)/(N-2)$. 
Equation \eqref{eqn-u} is equivalent to 
$$
(v^p)_\tau =  v_{ss} +  \alpha \, v^p -  \bar \alpha \, v, \quad \bar \alpha = \frac {(N-2)^2}4 ,\quad \alpha = \frac p{p-1} 
= \frac {N+2}4 
$$
or equivalently 
$$
\bar \alpha^{-1} (v^p)_\tau = \bar \alpha^{-1}  v_{ss} +  \alpha\,  \bar \alpha^{-1} \, v^p -   v.
$$
Setting  $v =  \lambda \bar v$, we find (after multiplying the above equation by $\lambda^{-1}$) that 
$$
\bar \alpha^{-1} \lambda^{p-1}\, (\bar v^p)_\tau= \bar \alpha^{-1}  v_{ss} + 
 \alpha\, \lambda^{p-1}\,  \bar \alpha^{-1} \,  v^p -   v.
$$
Choosing $\lambda$ so that $ \alpha\, \lambda^{p-1}\,  \bar \alpha^{-1} =1$ we finally conclude the following equation for $\bar v$ 
(which we denote again  by $v$)
\begin{equation}\label{eqn-v}
\alpha^{-1} (\ v^p)_\tau= \bar \alpha^{-1}  v_{ss} +   v^p -   v. 
\end{equation}

A  self-similar solution $U(r,t)$ of \eqref{eqn-u} given by \eqref{eqn-ss} corresponds to a traveling wave solution 
$V:=v(x- \beta \tau)$  of \eqref{eqn-v}. It follows that   $v$ satisfies the differential equation
\begin{equation}\label{eqn-vv}
\bar \alpha^{-1}  v_{ss} + \beta  \, (p-1) \,  v^{p-1} v_s +   v^p -   v =0. 
\end{equation}
We notice that the cylindrical solution $\C$ of equation \eqref{eqn-u} given by \eqref{eqn-cyl2} now corresponds to the constant solution  $v=1$ of equation \eqref{eqn-vv}. 
To linearize \eqref{eqn-vv} around the constant solution $v=1$,  we set $v:=1+w$ and we find that
$w$ satisfies the differential equation
\begin{equation}\label{eqn-w}
\bar \alpha^{-1}  \,   w_{ss} + \beta \, (p-1)  \,  (1+w)^{p-1} w_s +   (1+w)^p -   (1+w) =0
\end{equation}
or equivalently,  since $(p-1)\bar \alpha = N-2$,
\begin{equation}\label{eqn-ww2}
w_{ss} + \beta \, (N-2)  \,  (1+w)^{p-1} w_s +  \frac{N-2}{p-1} [ (1+w)^p -   (1+w) ]  =0.
\end{equation}
The linearized operator of \eqref{eqn-w} around $w=0$ is
\begin{equation}\label{eqn-w2}
L_\beta  w:=  w_{ss} + \beta\, (N-2)     w_s +  (N-2) \, w=0.
\end{equation}
We may write \eqref{eqn-ww2} as 
\begin{equation}\label{eqn-w9}
w_{ss} + \beta  \, (N-2)    w_s +   (N-2) \, w = f
\end{equation}
with
\begin{equation}\label{eqn-w10}
f := - (N-2) \, \left ( \beta \,  [(1+w)^{p-1} -1] \, w_s + \frac 1{p-1} [  (1+w)^p - 1 - p\, w ] \right ).
\end{equation}
Observe that
$$f= - \frac{(N-2)}{p} \, \left ( \beta \,  \phi'(s) + \frac p{p-1} \phi(s)\right ),$$
where
\begin{equation}\label{eqn-phit}
\phi(w):= (1+w)^p -1 - p \, w = c_p \, w^2 + O(w^3), \qquad \mbox{as} \,\, w \to 0.
\end{equation}
Since $1+w \geq 0$, we have $w \geq -1$ always. We observe that $\phi(w)$ is a convex function of $w$,
since $p >1$, and that its only  local minimum on $[-1,+\infty)$ is attained at $w=0$ where $\phi(0)=0$. 
Hence, 
$$\phi(w) \geq 0  \qquad \mbox{for all } \,\, w \in [-1,+\infty).$$
 
We next look  for solutions of \eqref{eqn-w2} of the form $w(s) = C\, e^{-\gamma s}$. It follows that $\gamma$ satisfies 
equation \eqref{eqn-eqgamma} and its roots $\gamma_i$, $i=1,2$ are given by \eqref{eqn-gamma}. 
The roots $\gamma_i$ are  real  (which give non-oscillating solutions $w$) iff 
$$ \beta^2 (N-2)^2  - 4 (N-2)  \geq 0  \quad \mbox{or} \quad  \beta \geq \beta_0:= \frac{2 }{\sqrt{N-2}}.$$
Notice that 
$$\beta_1:= \frac{N+2}{2(N-2)} \geq  \beta_0 :=  \frac{2 }{\sqrt{N-2}}$$
for all $N \geq 3$ and that $\beta_1=\beta_0$ iff $N=6$. 
Also,  if $\beta \geq \beta_0$, then $\gamma_2 \geq  \gamma_1$ and $\gamma_2=\gamma_1$ iff $\beta=\beta_0$.  

\subsection{Second-order asymptotics of the smooth self-similar solutions $U_{\beta,B}$}
\label{sec-Ub}
Our goal in this subsection   is to prove Theorem \ref{thm-ss}. Assume from now on that $\beta > \beta_0$, so that $\gamma_2 > \gamma_1$. Perform the cylindrical change of coordinates \eqref{eqn-cylcor} for $U_{\beta,B}$ and denote by $w(s)$ as above the perturbation of our solution in cylindrical coordinates  from the cylinder $v(s) \equiv 1$ when $s\to +\infty$. Note that the smoothness of our radial solution $U_{\beta,B}$ at the origin implies $w(s) \sim -1$ as $s\to -\infty$. We may express the solution $w$ of \eqref{eqn-ww2}  using the variation of parameters formula as
$$w(s) = - e^{-\gamma_1 s} \, \int_{-\infty}^s \frac{e^{-\gamma_2 t}}{W(t)} \, f(t) \, dt +  
 e^{-\gamma_2 s} \, \int_{-\infty}^s \frac{e^{-\gamma_1 t}}{W(t)} \, f(t) \, dt$$
 were $W(t)$ denotes the Wronskian determinant of the solutions $e^{-\gamma_1 t}$, $e^{-\gamma_2 t}$
 of the homogeneous equation and is equal to $W(t)=(\gamma_1 - \gamma_2) \, e^{-(\gamma_1+\gamma_2) t}$.
 It follows that
$$
w(s) = \frac 1{\gamma_2-\gamma_1}  \left ( e^{-\gamma_1 s} \, \int_{-\infty}^s e^{\gamma_1 t} \, f(t) \, dt -   
 e^{-\gamma_2 s} \, \int_{-\infty}^s e^{\gamma_2 t} \, f(t) \, dt \right ). 
$$ 
We conclude that
\begin{equation}\label{eqn-w3}
w(s) =  - C_N  \left ( A_1 \, e^{-\gamma_1 s} \, \int_{-\infty}^s e^{\gamma_1 t} 
\, \phi(t) \, dt 
-    
 A_2\, e^{-\gamma_2 s} \, \int_{-\infty}^s e^{\gamma_2 t} \, \phi(t) \, dt \right ),
\end{equation} 
with
\begin{equation}\label{eqn-Ai}
A_i :=  \frac{p}{p-1}-\beta \, \gamma_i, \,\,  i=1,2, \quad C=C(\beta,N), \quad 
C_N:=\frac {(N-2)}{p(\gamma_2-\gamma_1)} >0.
\end{equation}
Set $$I_i(s) := \int_{-\infty}^s e^{\gamma_i t}  \, \phi(t) \, dt \qquad \mbox{and} \qquad I_i =  \int_{-\infty}^{+\infty} e^{\gamma_i t}  \, \phi(t) \, dt \leq +\infty,$$
and recall that \eqref{eqn-phit} holds.

\begin{lem} The following hold:
\begin{itemize}
\item $A_1 > A_2$ for $\gamma_2 > \gamma_1$ and $A_1=A_2$ iff $\gamma_1=\gamma_2$ (or equivalently iff $\beta=\beta_0$). 
\item For $\beta > \beta_1$, we have $A_2 < 0 < A_1$. 
\item For $\beta = \beta_1$, we have
$A_1 =0$, $A_2 <0$ if $N < 6$ and $A_1  >0$, $A_2 =0$ if $N >6$
and $A_1=A_2=0$ if $N=6$ (in this case $\gamma_1=\gamma_2$). 
\item For $\beta_0 < \beta < \beta_1$ we have: $A_1 >A_2 >0$,   if $N >6$
and $A_2 <A_1 <0$,   if $N <6$.
\end{itemize}
\end{lem}

\begin{proof}
We have
$$A_i(\beta) = \frac{N-2}2 \, \left  ( \frac p2 - \big ( \beta^2 \mp \beta\, \sqrt{ \beta^2  -  \frac 4 {N-2}} \big ) \right ),
\quad i=1,2$$ 
hence by direct calculation
$$A_i'(\beta) =  \pm \frac{N-2}2 \, \frac { \big ( \beta  \mp  \sqrt{ \beta^2  -  \frac 4 {N-2}} \big )^2}
{\sqrt{ \beta^2  -  \frac 4 {N-2}}}.$$
Hence, $A_1(\beta)$ is increasing and $A_2(\beta)$ is decreasing. Also, 
$A_1(\beta_1) >0$ if $N >6$ and $A_1(\beta_1) =0$ if $N \leq 6$. Similarly, $A_2(\beta_1) =0$ if $N \ge6$ and $A_2(\beta_1) < 0$ if $N < 6$. Also, $A_i(\beta_0) >0$ if $N>6$ and  $A_i(\beta_0)=0$ if $N=6$ and 
$A_i(\beta_0) <0$ if $N<6$.  Hence:

\begin{itemize}
\item If $\beta > \beta_1$, then $A_1(\beta) > A_1(\beta_1) \geq 0$ and $A_2(\beta) < A_2(\beta_1) \leq 0$.
\item  If $N >6$, $\beta_0 < \beta < \beta_1$, we have $A_1(\beta) > A_1(\beta_0)> 0$,
$A_2 (\beta)  > A_2(\beta_1)=0$.
\item  If $N<6$, $\beta_0 < \beta < \beta_1$, we have $A_1(\beta)< A_1(\beta_1)=0$ and 
$A_2(\beta)  < A_2(\beta_0)<0$.
\item  If $N=6$, $\beta_0=\beta_1$ and $A_1 =A_2=0$, if $\beta=\beta_1=\beta_0$.  
\end{itemize}
\end{proof}

\begin{lem}\label{lem-www} Assume  $\beta > \beta_0$. We have
\begin{equation}\label{eqn-linfty1}
|w(s)| \leq C\, e^{-\gamma_1 s}, \qquad \mbox{for all} \,\,\, s \geq s_0
\end{equation}
for some constant $C$ depending on $N$ and $\beta$. Moreover,  if $N \le 6$ and $\beta_0 < \beta <  \beta_1$  there exists an $s_0$ such that $w(s) \ge 0$ for all $s\ge s_0$. In all other cases $w(s) <  0$ for all $s \in \R$.
\end{lem}

\begin{proof}
By  \eqref{eqn-w3} and \eqref{eqn-phit} we have 
$$
w(s) = - C_N \left ( A_1 \, e^{- \gamma_1 s} \,  I_1(s) - A_2 \, e^{- \gamma_2 s} \,  I_2(s) \right ),  
$$
as $s \to +\infty$, where $A_i$ are given by \eqref{eqn-Ai}, $C_N >0$   and $\lim_{s \to +\infty} w(s) = 0$. 
The case $A_1 = 0$ happens only when $\beta = \beta_1$ and  the solution is then explicit (Barenblatt solution). Therefore, we may assume from now on that $$A_1 \neq 0.$$
As $s \to \infty$,  
\begin{equation}\label{eqn-fw1}
w(s) =  - C_N \left (  \int_{-\infty}^s  \big ( A_1 e^{-\gamma_1 \, (s-t)}  - A_2 \, e^{- \gamma_2 (s-t)} \big )\, \phi(t)\, dt  \right )
\end{equation}
Moreover, 
$$\phi(t) \geq 0, \quad \forall t \qquad \mbox{and} \qquad e^{-\gamma_1(s-t)} \geq e^{-\gamma_2(s-t)}, \quad \forall t \leq s.$$
Recall that $A_2 <  A_1$ always when $\beta > \beta_0$. Note that $A_1 > 0$ in all cases except when $N \le 6$ and $\beta_0 < \beta \le \beta_1$. Since $A_2 < A_1$, it follows from \eqref{eqn-fw1} that when $A_1 >0$
 we have $w <0$ for all $s$ and 
\begin{equation}\label{eqn-fw3}
M_1
 \int_{-\infty}^s  e^{\gamma_1 t} \,   \phi(t) \, dt \leq e^{\gamma_1 s} \, |w(s)| \leq  M_2  
 \int_{-\infty}^s  e^{\gamma_1 t} \,   \phi(t) \, dt,
\end{equation}
with  $M_1:=A_1-|A_2| >0$ and $M_2:=A_1+|A_2| < \infty$. 

We claim that in the case when $N \le 6$ and $\beta_0 < \beta  < \beta_1$ there exists an $s_0$ such that  $w(s) \ge 0$, for all $s\ge s_0$. In order to see that, we multiply equation \eqref{eqn-w9} by $e^{\gamma_1 s}$ and integrate it over 
$(-\infty,s]$. 
After integration by parts, using that $-\gamma_1$ is a  solution to the characteristic equation for \eqref{eqn-w2} 
and that  $\gamma_1 + \gamma_2 = \beta(N-2)$,  we obtain
\begin{equation}
\label{eq-diff-ineq}
w_s + \gamma_2 \, w = - C_N \, e^{-\gamma_1 s}\, A_1\, \int_{-\infty}^s  \phi(t)\,  e^{\gamma_1 t}\, dt - C_N\, \phi.
\end{equation}
Since $A_1 <0$ in this case and $C_N\, \phi \leq c\, w^2$ for $s \geq s_0$ (by \eqref{eqn-phit}) we conclude 
\begin{equation}
\label{eq-touse}
w_s + \gamma_2 \, w \ge C_N \, |A_1|  e^{-\gamma_1 s} I_1(s) - c\, w^2
\end{equation}
for $s\ge s_0$. The above yields   that if ever $w(s_0) = 0$ for some $s_0$, then $w_s |_{s=s_0} \ge 0$, implying that 
$w(s) \ge 0$ for all $s \ge s_0$. Therefore we have two possibilities, either there exists an $s_0$ such that $w(s) \ge 0$ for all $s\ge s_0$, or $w(s)  <  0$ for all $s$. In the latter case, using that $\lim_{s\to \infty} w(s) = 0$, we can choose a tiny $\epsilon > 0$ so that for $s \ge s_0$
$$w_s + (\gamma_2 - \epsilon) \, w \ge 0,$$
implying
$$(w e^{(\gamma_2-\epsilon)s})_s \ge 0.$$
Since $\gamma_2 > \gamma_1$  we conclude  
$$|w(s)| \le C e^{-\gamma_1 s}, \qquad  s\ge s_0.$$
This would immediately imply that \eqref{eqn-linfty1} holds.  Since $A_1 < 0$ this would mean $w(s) > 0$ for $s$ sufficiently large, which contradicts our assumption that $w(s) \le 0$ for all $s$.  We conclude that the first possibility always holds,
namely $w(s) \ge 0$ for all $s \ge s_0$. Since $A_2 < A_1 <0$, it follows from \eqref{eqn-fw1} that 
\begin{equation}
\label{eq-case2}
e^{\gamma_1 s}\,  |w| \le C_N \, |A_1| \, \int_{-\infty}^s  \phi(t)\,  e^{\gamma_1 t}\, dt.
\end{equation}

We will now show that in all cases  \eqref{eqn-linfty1} holds.  By \eqref{eqn-fw3} (holding when $A_1 >0$) and 
\eqref{eq-case2} (holding when $A_1 <0$) it is sufficient to prove that   
$$I_1:= \int_{-\infty}^{+\infty} e^{\gamma_1 t}  \, \phi(t) \, dt < +\infty.$$
Indeed, assume that $I_1=+\infty$ and choose $s_0$ sufficiently large  that both 
 \eqref{eqn-phit}  and 
$$\int_{-\infty}^{s_0} e^{\gamma_1 t}  \, \phi(t) \, dt \leq \int_{s_0}^s e^{\gamma_1 t}  \, \phi(t) \, dt
\qquad \mbox{for} \,\,\, s >>1$$ 
 hold. 
By \eqref{eqn-fw3} and \eqref{eq-case2}, 
$$|w(s)| \leq C\, e^{- \gamma_1 s}\, \int_{s_0}^s e^{\gamma_1 t}  \, w^2(t) \, dt  \qquad \mbox{for} \,\,\, s >>1$$
for some positive constant $C >0$. 
We conclude that  $J_{s_0}  (s):=\int_{s_0}^s   e^{\gamma_1 t}  \, \phi(t) \, dt$ satisfies 
$$ J_{s_0}'(s) \leq  2 C\, e^{-\gamma_1 s} \, J_{s_0}(s)^2 \qquad \mbox{for}\,\,\, s >>1$$
from which, after we integrate on $[s,+\infty)$ and use that $J_{s_0} := \lim_{s\to \infty} J_{s_0} (s) =+\infty$, we obtain the lower bound
$$J_{s_0} (s) \ge c  \, e^{\gamma_1 s}$$ for some absolute $c >0$. Since for $s >>1$  we either have $w \le 0$ and $A_1 > 0$ or $w \ge 0$ and $A_1 < 0$,  this lower bound and \eqref{eq-diff-ineq} imply
$$|w_s| + (\gamma_2 + \epsilon) |w| \ge C_N e^{-\gamma_1 s} |A_1| \, J_{s_0}(s) \ge c >0$$ 
yielding a contradiction since  $\lim_{s\to \infty} w(s) = 0.$

\smallskip
We will finish by showing that $w(s) < 0$ for all $s \in \R$ in the cases  stated in the lemma. 
As shown in \eqref{eqn-Ai}, the  constant $C_N$ in front of $\phi(s)$ in \eqref{eqn-w3} is positive. 
In the case when $\beta > \beta_1$ and  $N \geq 3$   we have $A_2 < 0 < A_1$ and therefore \eqref{eqn-w3} implies that $w(s) < 0$ for all $s$. In the case  $\beta_0 \le \beta < \beta_1$ and $N \ge 6$,   since $A_1 > A_2 > 0$ and $\gamma_1 < \gamma_2$,  we have $A_1 e^{-\gamma_1(s-t)} - A_2 e^{-\gamma_2(s-t)} > 0$. Equation \eqref{eqn-w3}  implies again  that $w(s) < 0$ for all $s$.

\end{proof}

\begin{lem}\label{lem-www33} If $A_1 \neq 0$, then 
\begin{equation}\label{eqn-I1}
w(s) = -  C_N \,  A_1 \, I_1 \, e^{- \gamma_1 s} (1+ o(1))
\end{equation}
with $I_1:= \int_{-\infty}^{+\infty} e^{\gamma_1 t} \, \phi(t) \, dt$ satisfying $0 < I_1 < \infty$. 
\end{lem}

\begin{proof}
We will use \eqref{eqn-fw1}.    We first observe that by \eqref{eqn-phit} and \eqref{eqn-linfty1} we have 
$$e^{- \gamma_2 s} \,  I_2(s) \leq  e^{- \gamma_2 s} \, I_2(s_0) + 2 c_p  \, e^{- \gamma_2 s} \, \int_{s_0}^s e^{\gamma_2t} w^2(t)\, dt \leq  e^{- \gamma_2 s} \, I_2(s_0) + C\, e^{-2\gamma_1 s}.$$
Since  $I_1 \neq 0$ and $\gamma_2 > \gamma_1$,  it follows from \eqref{eqn-fw1} that as $s \to \infty$, 
\eqref{eqn-I1} holds.
\end{proof}
The above discussion leads to the following Proposition.

\begin{prop}
\label{prop-cyl}
Let $m = (N-2)/(N+2)$, $N \ge 3$, $\beta_0 := 2/\sqrt{N-2}$ and $\beta_1 = 1/(2m)$. We have the following:
\begin{itemize} 
\item For $N \geq 6$ and  $\beta >  \beta_0$ or $2 < N < 6$ and $\beta  >  \beta_1$
 the solution to \eqref{eqn-ww2} admits  the slow behavior 
$w(s) =- B \, e^{-\gamma_1 s} (1+o(1))$ with  $B>0$. 

\item For $2 < N < 6$ and $\beta_0 < \beta  < \beta_1$ 
the solution to \eqref{eqn-ww2} admits  the slow behavior  
$w(s) =  B \, e^{-\gamma_1 s} (1+o(1)) $ with  $B >0$.

\item  For $2 < N \le 6$ and $ \beta= \beta_1$,  the solution to \eqref{eqn-ww2} admits  the fast  behavior 
$w(s) =- B \, e^{-\gamma_2  s} (1+o(1))$, with  $\gamma_2=\gamma_2(\beta_1)=2$, $B >0$. 
\end{itemize}
\end{prop}
 
\begin{proof} 
In the cases where $A_1 \neq 0$, the Proposition  follows  from \eqref{eqn-I1}. In the cases where $A_1=0$, we have $\beta=\beta_1$ and the solution is given in closed form (Barenblatt solution) and admits  the fast behavior $w(s) =- B \, e^{-2s} (1+o(1))$. 
\end{proof}

\begin{proof}[Proof of Theorem \ref{thm-ss}]
The proof of Theorem  \ref{thm-ss} immediately follows from  Proposition \ref{prop-cyl} once we write $v(s) = 1 + w(s)$ and express everything in polar coordinates on $\mathbb{R}^N$.
\end{proof}

\subsection{Asymptotics of the singular  self-similar solutions $V_{\beta,B}$}
Our goal in this subsection is to prove Theorem \ref{thm-ss-sing}.  Before showing the precise asymptotics  of those singular self-similar solutions at infinity we will first briefly comment on their  existence. Since this is pretty standard we will omit the details and give the references in which the details can be found.

\begin{lem}
\label{lem-V-B}
For every $\beta > \beta_1$ and $T >0$,  there exists  one parameter family of self-similar solutions 
$V_{\beta,B}(x,t) = (T-t)^{\alpha} g_{\beta,B}(x(T-t)^{\beta})$ with profile function  $g_{\beta,B}(y)$ satisfying  \eqref{eqn-orig} with $B >0$ and $K_B >0$ (depending on $B$) so that 
\begin{equation}\label{eqn-cccc}
g(y)= \left ( \frac{C^*}{|y|^2} \right )^{1/n} (1 + o(1)), \quad \mbox{as} \,\, |y| \to +\infty
\end{equation}
\end{lem}

\begin{proof}
We will  show the existence of a radial solution $g_{\beta,B}(r)$, $r=|y|$  of equation \eqref{eqn-ell} on $\R^N$  such that $g_{\beta,B}(r)$ satisfies \eqref{eqn-orig} and \eqref{eqn-cccc}. Omit for simplicity the subscripts $\beta, B$ in the proof of this Lemma, but keep in mind that $\beta > 0$ has been fixed. Rewrite the equation \eqref{eqn-ell}  in the following form
\begin{equation}
\label{eq-other-form}
\frac 1m r^{1-N} (r^{N-1} g(r)^{m-1} g'(r))' + \beta  y g'(r) + \alpha g(r) = 0.
\end{equation}
As in \cite{V} we  introduce the following change of variables
$$y = e^s, \qquad X(s) := \frac{r g'}{g}, \qquad Y(s) = y^2 g^{1-m}.$$
Then \eqref{eq-other-form} is equivalent to the following autonomous system of ODEs
\begin{equation}
\begin{split}
\label{eq-aut-sys}
\dot{X} &= (2-N) X - mX^2 - m(\alpha + \beta X)Y \\
\dot{Y} &= (2+ (1-m)X)Y. 
\end{split}
\end{equation}
Set $\theta:= - \alpha /\beta$. This autonomous system has a local solution with
$$X(-\infty) = -\theta, \qquad Y(s) \sim e^{s(-\theta (1-m) + 2)}, \quad \mbox{as} \,\,\, s\to -\infty.$$
These are equivalent to saying that initial data satisfy $g(r) \sim r^{-\theta}$ as $y\to 0$. Take any self-similar profile $f_{\beta,B}$ of a smooth self-similar solution $U_{\beta,B}$ (as in Theorem \ref{thm-ss}).  Then $\lim_{r\to 0} g(r) > f_{\beta,B}(0)$. We claim this implies 
\begin{equation}
\label{eq-lower-g}
g(r) > f_{\beta,B}(r)
\end{equation}
for all $r$ as long as the solution $g$ exists. Indeed, if that were not true, there would exist an $r_0$ so that (since our solutions are radially symmetric)  $g(r_0) = f_{\beta,B}(r_0)$. Since both, $g$ and $f_{\beta,B}$ are weak solutions  of the same 
elliptic equation \eqref{eqn-ell},   by the uniqueness of the  Dirichlet problem on the  ball
$B_{r_0}(0)$ (which holds since $\beta > \beta_1$) we would have $g \equiv f_{\beta,B}$ on the same ball. This  is impossible
since $f_{\beta,B}$ is a smooth solution at the origin unlike the singular solution $g$.

On the other hand the Aronson Benil\'an inequality (\cite{AB}), applied to the ancient solution $V_{\beta,B}$ implies $\Delta V_{\beta,B}^m \le 0$. This means that $\Delta g^m \le 0$, which (if $|y| = r$) is equivalent to $(r^{N-1}(g^m)_r)_r \le 0$.  After integrating this inequality twice from some fixed $r_0 > 0$ to $r$  we get
\begin{equation}
\label{eq-upper-g}
g^m(r) \le C(1+ r_0^{2-N}).
\end{equation}
Combining the estimates \eqref{eq-lower-g} and \eqref{eq-upper-g} yields that our solution $g$ 
remains strictly positive and bounded for all  $r >0$  and therefore it defines the global solution to \eqref{eq-other-form}.

Almost the same phase-plane analysis as   in \cite{V} (see chapter 5)  implies  that by choosing the right orbit  
the solution  $g$ admits  the  cylindrical behavior 
$g(r)=(C^*|r|^{-2})^{1/n} (1 + o(1))$, as $r \to +\infty$. To see that we find that the critical points of our system \eqref{eq-aut-sys} are
$$E := (0,0), \qquad C:= (-(N+2), 0), \qquad D:= (-2/(1-m), (N-2)/m).$$
The only difference from the analysis in \cite{V} is that we need to exclude the case that the orbit ends at the critical point $C$. If that were to happen, as in \cite{V} we would get that our solution had the spherical behavior at infinity and therefore was in $L^1(\mathbb{R}^N)$. We argue by contradiction that this is not possible. Assume 
$V_{\beta,B}(x,t) = (T-t)^{\alpha} g_{\beta,B}(x(T-t)^{\beta})$ and $g_{\beta,B}$  satisfies \eqref{eqn-orig} and $g(y) \sim |y|^{-(N+2)}$ as $|y| \to\infty$. Let $U_{s}(x,t) := (T-t)^{\alpha} f_s(x(T-t)^{\beta})$ be  the spherical solution that becomes extinct at time $T$ as well. By the $L^1$ contraction principle we have
$$\int_{\mathbb{R}^N} |V_{\beta,B}(x,t) - U_s(x,t)|\, dx \le \int_{\mathbb{R}^N} |V_{\beta,B}(x,0) - U_s(x,0)|\, dx \le C < \infty,$$
where $C > 0$ is a uniform constant. This implies that 
$$\int_{\mathbb{R}^N} |g_{\beta,B}(y) - f_s(y)|\, dy \le C (T-t)^{\beta N - \alpha} \to 0,$$
as $t\to T$, forcing $g_{\beta,B} \equiv f_s$, which is impossible. We have used here our assumption that $\beta > \beta_1$ (or equivalently $\beta N > \alpha$). 
Similarly as in \cite{V} we conclude the orbit must end at $D$, hence as $s\to \infty$ we must have $X \to -2/n$, which implies the asymptotic  behavior \eqref{eqn-cccc}. 
\end{proof}

By Lemma \ref{lem-V-B} we already know the existence of the singular self-similar solutions 
$V_{\beta,B}$ for which their profile $g_{\beta,B}$ satisfies  \eqref{eqn-orig} and  
\eqref{eqn-cccc}. To complete the proof of Theorem \ref{thm-ss-sing} it is therefore enough to show \eqref{eqn-asympt2}. For this purpose assume $\beta \ge \beta_1$. 
\begin{proof}[Proof of Theorem \ref{thm-ss-sing}]
We adopt the same notation as in subsection \ref{sec-Ub}. Using the variation of parameters formula we can write $w$ as
$$w(s) = \frac{1}{\gamma_2 - \gamma_1}\left( e^{-\gamma_1 s}\, \int_{s_0}^s e^{\gamma_1 t} f(t)\, dt - e^{-\gamma_2 s}\int_{s_0}^s e^{\gamma_2 t} f(t)\, dt\right).$$
Note that we are not able to integrate from $-\infty$ as in the proof of Theorem \ref{thm-ss}  but instead do so from some finite $s_0$, the reason being that the singular behavior of $V_{\beta,B}$ at the origin implies that the above integrals are near  $-\infty$.
Integration by parts yields 
$$w(s) = C_N \left( e^{-\gamma_1 s} \big( \beta \phi(s_0) e^{\gamma_1 s_0} - A_1 \int_{s_0}^s e^{\gamma_1 t}\phi(t)\, dt\big) - e^{-\gamma_2 s} \big( \beta \phi(s_0) e^{\gamma_2 s_0} + A_2 \int_{s_0}^s e^{\gamma_2 t}\phi(t)\, dt \big)\right).$$
Recall that for $\beta \ge \beta_1$ we have $A_2 \le 0 \le A_1$. By Corollary \ref{cor-mon} we have $w(s) > 0$ for all $s\in (-\infty,\infty)$ in the considered case. This, together with $\phi \ge 0$ (as we showed  in the proof of Lemma \ref{lem-www}), implies
\begin{equation}\label{eqn-ws5}
0 \le w(s) \le C(s_0) e^{-\gamma_1 s} \qquad \mbox{for all} \,\,\, s.
\end{equation}
As in the proof of Lemma \ref{lem-www33}  we can  now argue that, if $I(s_0) \neq 0$ for some $s_0 \in \R$, then 
$$w(s) =C_N \, I(s_0) e^{-\gamma_1 s} (1 + o(1)),$$
where $I(s_0) := \beta\phi(s_0)\,  e^{\gamma_1 s_0} - A_1 \int_{s_0}^{\infty} e^{\gamma_1 t}\phi(t)\, dt$. 
Since $w >0$ it follows that $I(s_0) \ge 0$. If $I(s_0) > 0$ for  some $s_0$  we are done.  Otherwise, we must have 
$$\beta e^{\gamma_1 s}  \phi(s) = A_1 \int_s^{\infty} e^{\gamma_1 t}\phi\, dt, \qquad \forall s \in (-\infty, +\infty)$$
implying that 
\begin{equation}
\label{eq-beh1}
\phi(s) = C e^{-\big(\gamma_1 + \frac{A_1}{\beta}\big) s}.
\end{equation}
On the other hand, $\phi(s) \sim w^p(s)$ as $s\to -\infty$  and, by  using the change between radial and cylindrical coordinates and the behavior $g_{\beta,B}(y) \sim |y|^{-\alpha/\beta}$ as $|y| \to  0$,  we obtain 
$w(s) \sim e^{-\left(\frac{\alpha}{\beta} - \frac{1}{p-1}\right)\, s}$ as $ s\to -\infty.$
This together with \eqref{eq-beh1} and the definition of $A_1$ would imply that 
$$\left(\frac{\alpha}{\beta} - \frac{1}{p-1}\right) p = \gamma_1 + \frac{A_1}{\beta}=\frac p{(p-1)\beta}.$$
A direct calculation shows that this  is equivalent to $b=-4/(N+6)$, which is impossible since $b$ is always positive. 
This means that $I(s_0) \neq 0$ for some $s_0$ finishing our proof. 

\end{proof}


\section{Cylindrical Behavior of evolving metrics  at infinity}
\label{cyl-beh}

Recall that the cylindrical self-similar solution to \eqref{eqn-fd} is given by \eqref{eqn-cyl2} and  becomes  extinct at time $T$. Both self similar solutions  $U_{\beta,B}$ and $V_{\beta,B}$ whose second-order asymptotics have been discussed in section \ref{sec-asymp} 
have cylindrical behavior at infinity and they both become extinct at the time when their cylindrical tails become extinct. These suggest that the cylindrical tail of any solution $u$ that satisfies \eqref{eqn-u0infty} will become extinct at time $T$, as shown   in the following proposition.

\begin{prop}\label{prop-tail}
Let $u$  be a nonnegative weak  solution of \eqref{eqn-u} with initial data $u_0 \in L^\infty (\R^N)$ satisfying \eqref{eqn-cyl2} with $C^*$ given by \eqref{eqn-cstar} and $n=1-m$. 
Then, for  every $t\in [0,T)$ we have
$$u(x,t) = \left(\frac{C^*(T-t)}{|x|^2}\right)^{\frac 1n} (1+o(1)).$$
\end{prop}

Before we prove this Proposition we will show that the extinction time of any  solution $u$ with initial data  $u_0 \in L^1_{loc}(\R^N)$ satisfying \eqref{eqn-cyl2}
is at least $T$.

\begin{lem}
\label{lem-vanishing-time}
Let $u$ be a  solution   of \eqref{eqn-u}  as in Proposition \ref{prop-tail}. Then,  its extinction time  $T^*$ satisfies  $T^* \ge T$.
\end{lem}

\begin{proof}
It is well known that bounded solutions to  \eqref{eqn-u} with $u_0 \geq 0$ 
are $C^\infty$ smooth and strictly positive  up to their extinction time $T^*$. 
Hence, we may assume without loss of generality,   that the initial data $u_0$ 
are strictly positive and continuous  on $\R^N$. 
Let $\epsilon > 0$ be an arbitrarily small positive number.  By the  asymptotics given by \eqref{eqn-cyl2} and the positivity 
and continuity of $u_0$, it follows that we can choose $k >0$ to
imply that $$u_0(x) \ge \left(\frac{C^* (T - \epsilon)}{|x|^{2}+k^2}\right)^{1/n} \qquad \forall x\in \mathbb{R}^N.$$
By  comparison with the Barenblatt solutions \eqref{eqn-ub}  we have
\begin{equation}
\label{eq-lower-T}
u(x,t) \ge \left(\frac{C^* (T - \epsilon - t)}{|x|^{2}+k^2  (T - \epsilon - t)^{- 2 \beta_1}}\right)^{1/n} \qquad \forall (x,t)\in \mathbb{R}^N \times [0,T-\epsilon)
\end{equation}
where $\beta_1 := (N+2)/(2(N-2))$. This implies that the extinction time
$$T^* \ge T - \epsilon.$$
Since $\epsilon > 0$ is arbitrary, let $\epsilon \to 0$ above to conclude the proof of the Lemma.
\end{proof}

\begin{proof}[Proof of Proposition \ref{prop-tail}]
Let $\epsilon > 0$ be arbitary.
By \eqref{eq-lower-T} we have  
$$|x|^{2/n} \, u(x,t)  \ge \left ( \frac{C^* \, (T-\epsilon-t) \, |x|^2}{|x|^2+k^2  (T-\epsilon-t)^{- 2\beta_1}} \right )^{1/n}.$$
If we let $|x| \to \infty$ above we get
$\liminf_{|x|\to\infty} \, |x|^{2/n} \, u(x,t)  \ge \big (C^*\, (T-\epsilon-t) \big )^{1/n}$
and  by  letting  $\epsilon \to 0$ we obtain 
\begin{equation}
\label{eq-below}
\liminf_{|x|\to\infty} \, |x|^{2/n} \, u(x,t)  \ge \big (C^*\, (T-t) \big )^{1/n}. 
\end{equation}
On the other hand, it is easy to check that the function 
$$B_{k^-}(x,t) := \left(\frac{C^*(T-t)}{|x|^2 - k^2(T-t)^{- 2\beta_1}}\right)^{1/n}$$
 solves the equation \eqref{eqn-fd} on  $\{|x| > k(T-t)^{-\beta_1}\}\times [0,T)$. Let $\epsilon > 0$.  By our assumption on $u_0$ there exists an $r_0$ so that for all $|x| \ge r_0$  we have 
\begin{equation}
\label{eqn-BB-00}
u_0(x) \le \left(\frac{C^* (T+\epsilon)}{|x|^2}\right)^{1/n} \le \left(\frac{C^* (T+\epsilon)}{|x|^2 - k^2 
(T+\epsilon)^{- 2\beta_1}}\right)^{1/n}.
\end{equation}
Choose $k$ sufficiently large  that $k\, (T+\epsilon)^{-\beta_1} \ge r_0$ and  
set   
$$B_{k^-}^\epsilon (x,t) :=\left(\frac{C^* (T+\epsilon-t)}{|x|^2 - k^2 
(T+\epsilon-t)^{- 2\beta_1}}\right)^{1/n},$$
which solves   equation \eqref{eqn-fd} for   $(x,t) \in \{|x| > k\, (T+\epsilon -t)^{-\beta_1}\}\times [0,T)$.  
The rescaled functions 
$$\tilde  u (y,t) := u( (T+\epsilon-t)^{-\beta_1}\, y,t) \quad \mbox{and}  
\quad \tilde  B_{k^-}^\epsilon (y,t) := B_{k^-}^\epsilon ( (T+\epsilon-t)^{-\beta_1}\, y,t)$$
satisfy the equation 
\begin{equation}
\label{eq-tilde-u-00}
 \tilde{u}_t  = (T+\epsilon-t)^{ 2\beta_1} \Delta \tilde{u}^m + \beta_1 \, (T+\epsilon -t)^{-1}  \,  y \cdot  {\nabla \tilde u}
\end{equation}
on $Q_k:=\{|y| >  k \}  \times [0,T)$ and 
$\tilde{u}(y,0) \le \tilde{B}_{k^-}^{\epsilon}(y,0)$
for all $|y| >  k$, from \eqref{eqn-BB-00} and the choice $k\, (T+\epsilon)^{-\beta_1} \ge r_0$. Note also that for every $t\in [0,T)$ and for every $y_0$ with  $|y_0| = k$
$$\lim_{y\to y_0} \tilde{u}(y,t) \le \lim_{y\to y_0} \tilde{B}_{k^-}^{\epsilon}(y,t),$$
since the right hand side is infinite. 
The comparison principle applied to \eqref{eq-tilde-u-00}  on $Q_k$ yields that 
$\tilde{u}(y,t) \le \tilde{B}_{k^-}^{\epsilon}(y,t)$ on  $Q_k$, or equivalently 
 $u(x,t) \le B_{k^-}^{\epsilon}(x,t)$ on $\{|x| > k\, (T+\epsilon -t)^{-\beta_1}\}\times [0,T)$ implying the bound
$\limsup_{|x|\to\infty} |x|^{2/n} \, u(x,t)  \le \left( C^* \, (T+\epsilon - t) \right)^{1/n}.$
Letting $\epsilon \to 0$ yields
$$\limsup_{|x|\to\infty} |x|^{2/n} \, u(x,t)  \le \left( C^* \, (T- t) \right)^{1/n},$$
which  together with \eqref{eq-below} implies the statement  of the proposition.  
\end{proof}

\section{Extinction profile of  solutions to the Yamabe flow}\label{sec-conv}

We assume in this section that
\begin{itemize}
\item
either $N \ge 3$ and $\beta \ge \beta_1:=1/(2m)$, or 
\item
$N \ge 6$ and $\beta < \beta_1:=1/(2m)$ with $\beta > \beta_0:={2}/{\sqrt{N-2}}$.
\end{itemize}
Our main goal   is to prove Theorems \ref{thm-conv} and \ref{thm-conv2}. 

The assumption $u(x,0) \le U_{\beta,B_1}(x,0)$,  for some $B_1 > 0$,  which  is assumed to hold in both Theorems, and  
 the comparison principle imply the upper bound 
\begin{equation}
\label{eq-comp}
u(x,t) \le U_{\beta,B_1}(x,t) \qquad \mbox{on} \,\,\, \R^N \times (0,T). 
\end{equation}
In particular, $u$ vanishes at time $T$. The rescaled function $\bar u$ given by \eqref{eq-resc-left} satisfies the  rescaled equation \eqref{eq-resc100}.

\medskip

\subsection{The case  $\beta \ge \beta_1$}
In this section we will give the proof of  Theorem \ref{thm-conv}. We begin by  stating  the $L^1$-contraction property 
for  solutions to \eqref{eqn-u}, whose proof can be found in \cite{HP, V}.

\begin{lem}[$L^1$-contraction \cite{HP}, \cite{V}]
\label{lem-contraction}
For any two non-negative solutions $u_1$ and $u_2$ of \eqref{eqn-u} with initial data in $L^1_{loc}(\mathbb{R}^N)$,  
defined on a time interval $[0,T)$,    and any two times $t_1$ and $t_2$ such that $0 \le t_1 \le t_2 < T$, we have
$$\int_{\mathbb{R}^N} |u_1(x,t_2) - u_2(x,t_2)|\, dx \le \int_{\mathbb{R}^N} |u_1(x,t_1) - u_2(x,t_1)|\, dx.$$
\end{lem}
Observing  that 
$$\int_{\mathbb{R}^N} |u(x,t) - v(x,t) | \, dx = e^{(\beta N - \alpha)\tau} \, 
\int_{\mathbb{R}^N} |\bar{u}(y,\tau) - \bar{v}(y,\tau)|\, dy,$$ 
the  contraction property implies the following decay estimate. 

\begin{cor}
\label{cor-resc}
Let $u(\cdot,t), v(\cdot,t)$ be two solutions to \eqref{eqn-u} and let $\bar{u}(\cdot,\tau), \bar{v}(\cdot,\tau)$ be their rescalings, respectively. Then
\begin{equation}
\label{eq-exp-decay}
\int_{\mathbb{R}^N} |\bar{u}(y,\tau) - \bar{v}(y,\tau)|\, dy  \le e^{-(\beta N - \alpha)\, \tau} \int_{\mathbb{R}^N} |\bar{u}(y,0) - \bar{v}(y,0)|\, dy
\end{equation}
for all $ \tau\in (0,\infty).$
\end{cor}

\begin{proof}[Proof of Theorem \ref{thm-conv}]
By  \eqref{eq-comp} and \eqref{eqn-ss}  we have 
$$\bar{u}(y,\tau) \le f_{\beta, B_1}(y) \leq  f_{\beta, B_1}(0) \qquad (y,\tau)\in \R^N \times (0,\infty)$$
since the profile $f_{\beta,B_1}$ is decreasing in $|y|$. 
It follows that for any sequence $\tau_i\to \infty$,  the sequence of solutions 
 $\bar{u}_i(y,\tau) := \bar{u}(y,\tau_i+\tau)$ is uniformly bounded and hence 
 it is equicontinuous on compact subsets of $\R^N \times (-\infty,+\infty)$ 
 by  well-known equicontinuity result for solutions to fast diffusion equations (see  in \cite{DKB}). 
 Hence, by the    Arzela-Ascoli theorem  there exists a subsequence (still denoted by $\tau_i$)  such that
$\bar{u}_i \to  \bar{u}_{\infty}$  as $ i\to\infty$,  uniformly on compact subsets of $\R^n \times (-\infty,\infty)$.   

We will next show that, because  of our assumption  \eqref{eq-int},  
we have $\bar{u}_{\infty} \equiv f_{\beta,B}$,  with  $f_{\beta,B}$ denoting the profile function of the self-similar solution
$U_{\beta,B}$ defined by \eqref{eqn-ss}. 
To this end   we  apply \eqref{eq-exp-decay} to our solution $\bar{u}_i(y,\tau)$ and to the rescaled self-similar solution 
$f_{\beta,B}(y)$ to obtain 
\begin{equation}
\label{eq-let}
\int_{\mathbb{R}^N} |\bar{u}_i(y,\tau) - f_{\beta,B}(y)|\, dy \le e^{-(\beta N - \alpha)(\tau_i+\tau)} \int_{\mathbb{R}^N} |\bar{u}_0(y) - f_{\beta,B}(y)|\, dy.
\end{equation}
Note  that because of \eqref{eq-int}, we have  $\int_{\mathbb{R}^N} |\bar{u}_0(x) - f_{\beta,B}(x)|\, dx < \infty$. Let $i\to \infty$ in \eqref{eq-let} to conclude
$$\int_{\mathbb{R}^N} |\bar{u}_{\infty}(y,\tau) - f_{\beta,B}(y)|\, dy \le
\lim_{i\to \infty} \int_{\mathbb{R}^N} |\bar{u}_i(y,\tau) - f_{\beta,B}(y)|\, dy  = 0.$$
This implies $\bar{u}_{\infty} \equiv f_{\beta,B}$. 

We conclude that $\bar{u}(y,\tau)$  converges,  as $\tau\to\infty$, uniformly on compact subsets of $\R^N$  and also in the $L^1{(\mathbb{R}^N)}$ norm to the self-similar profile  $f_{\beta,B}$. The latter convergence is exponential and the exponential rate of convergence is at least   $e^{-(N\beta - \alpha)\tau}$.
\end{proof}

\medskip

\subsection{The case $\beta_0  < \beta < \beta_1$ and $N \ge 6$}

In this section we will give the proof of   Theorem \ref{thm-conv2}.  Let $\eta\in C^{\infty}_0(\mathbb{R})$ be a cut off  function such that $\eta(y) =1/2$ for $|y| < 1/2$ and $\eta(y) = 0$ for $|y| > 1$. Let $\eta_R(y) := \eta(y/R)$, $\eta_{\epsilon} := \eta({y}/{\epsilon})$ and $\eta_{R,\epsilon} := \eta_R(y) + \eta_{\epsilon}(y)$.
Note that  $|\Delta\eta_{R,\epsilon}| + |\nabla \eta_{R,\epsilon}|^2 \le C\, \epsilon^{-2}$  for ${\epsilon}/{2} \le |y| \le \epsilon$ and  $|\Delta \eta_{R,\epsilon}| + |\nabla\eta_{R,\epsilon}|^2 \le C\, R^{-2}$ for ${R}/{2} \le |y| \le R$. 
  
We start with the following weighted contraction result. We recall the weighted $L^1$ space  given by \eqref{eqn-wl1}
with  $p_0 \in (0,2m)$. 

\begin{lem}
\label{lem-contr-weight}
Let $\bar u, \bar v$  be any solutions  to  \eqref{eq-resc-left},  with initial data  $\bar{u}_0, \bar v_0$ respectively, 
satisfying $\bar u_0,\bar v_0 \leq f_{\beta,B_1}$. If 
$\max_{\mathbb{R}^N}|\bar{u}_0 -\bar v_0| \neq 0$ 
then 
$$\|(\bar{u} - \bar v_0)(\cdot,\tau) \bar{\C}^{p_0}\|_{L^1(\R^N)} < \|(\bar{u}_0 - \bar \C) \bar{\C}^{p_0}\|_{L^1(\R^N)}
\qquad \forall  \tau \ge 0.$$
\end{lem}

\begin{proof} 
The condition $\bar u_0, \bar v_0 \leq f_{\beta,B_1}$ implies that $\bar u(\cdot,\tau), \bar v (\cdot,\tau) \leq f_{\beta,B_1}$
 for all $\tau \geq 0$,
since $f_{\beta,B_1}$ is a steady state of  the rescaled equation \eqref{eq-resc-left}. 
Recall also that in the case $N \ge 6$ we have shown in Lemma \ref{lem-www} that $w(s) < 0$ for all $s \in \R$,  which is equivalent to 
\begin{equation}
\label{eq-bound-U}
f_{\beta,B_1}(y) < \bar \C(y) \qquad \forall x\in \mathbb{R}^N
\end{equation}
where $\bar{\C}(y) = \left( C^* |y|^{-2}\right)^{1/n}$  with $C^*$ the constant in \eqref{eqn-cstar}
corresponding to the cylindrical metric.

Set $q := |\bar u- \bar{v}|$. A standard application of Kato's inequality implies that 
\begin{equation}
\label{eq-q}
q_\tau  \le \Delta(a q) + \beta\, \mbox{div}\, (y \, q) + (\alpha - N\beta)\, q
\end{equation}
in the distributional sense, where
\begin{equation}
\label{eq-a}
a(y,\tau) := \int_0^1 \frac{m\,d\theta}{(\theta\bar{\C} + (1-\theta)\bar{v})^n}.
\end{equation}
By the bound $\bar v \leq f_{\beta,B_1}$ and \eqref{eq-bound-U} we have 
\begin{equation}
\label{eq-a}
a(y,\tau) >  \frac m{\bar C^n} = \frac{m\, |y|^2}{C^*}.
\end{equation}
Let $\eta_{R,\epsilon}$ be the cut off function introduced  above. Equation \eqref{eq-q} and integration by parts yield 
\begin{equation}
\begin{split}
\label{eq-mon100}
\frac{d}{d\tau}\int_{\mathbb{R}^N} q \, \eta_{R,\epsilon} \, \bar{\C}^{p_0}\, dy &\le \int_{\mathbb{R}^N} \big [a\, \Delta \bar{\C}^{p_0} - \beta\,\nabla\bar{\C}^{p_0} \cdot y + (\alpha - N \beta)\, \bar{\C}^{p_0} \big ]\, q\, \eta_{R,\epsilon} \, dy  \nonumber \\
&+ \int_{\mathbb{R}^N} \big [\Delta \eta_{R,\epsilon} \bar{\C}^{p_0} a  + 2a  \, \nabla\eta_{R,\epsilon}\cdot\nabla\bar{\C}^{p_0} - \beta y\cdot\nabla\eta_{R,\epsilon} \bar{\C}^{p_0} \big ]\, q\, dy. 
\end{split}
\end{equation}
A direct calculation   shows that 
$$\Delta\bar{\C}^{p_0} = \frac 14 (N+2) (C^*)^{p_0/n}\,  p_0\, \big ( N\, (p_0 - 2) + 2( p_0+2) \big )\,  |x|^{-2 -  (N+2)p_0/2}.$$
We see that $\Delta\bar{\C}^{p_0} < 0$ for $p\in (0, 2m)$. Hence by  \eqref{eq-a} \begin{equation}
\label{eq-lap}
a\, \Delta\bar{\C}^{p_0} <  \frac{N-2}{4C^*} p_0 \, \big ( N(p_0 - 2) + 2(p_0 + 2) \big )\, \bar{\C}^{p_0}.
\end{equation}
Furthermore,
\begin{equation}
\label{eq-rest}
-\beta \, \nabla\bar{\C}^{p_0} \cdot y + (\alpha - N\beta) \, \bar{\C}^{p_0} = \frac 14 \bar{\C}^{p_0} (2+ N + 2\beta\, (2 - N + 2\beta (2-N + (2+N)\, p_0)).
\end{equation}
Estimates \eqref{eq-lap} and \eqref{eq-rest} imply
$$a\, \Delta\bar{\C}^{p_0} - \beta \nabla\bar{\C}^{p_0}\cdot x + (\alpha - N\beta)\bar{\C}^{p_0} < K_N\, \bar{\C}^{p_0}$$
with
$$K_N:= \frac{-4 - 2\beta(N-2)^2 + N^2 + 2(-1 + \beta)(N^2-4) p_0 + (N+2)^2 p_0^2}{4(N-2)}.$$
We see that 
\begin{equation}
\label{eqn-strict}
a\, \Delta\bar{\C}^{p_0} - \beta \, y \cdot \nabla\bar{\C}^{p_0} + (\alpha - N\beta)\, \bar{\C}^{p_0} < 0,
\qquad \mbox{if} \,\,\,  p_0 \in [p_1,p_2]\end{equation}with 
 $p_1 := m\left(1 - \beta - \sqrt{\beta^2 - \frac{4}{N-2}}\right)$ and  $p_2 := m\left(1 - \beta + \sqrt{\beta^2 - \frac{4}{N-2}}\right)$.
 We will  choose
 $p_0:=p_2$. 
Since $\beta < \beta_1$ and $N  \ge 6$, it is easy to check that
$p_1  > 0$ and $p_2 < 2m$, so that $p_0\in (0,2m)$. With this choice of $p_0$ we conclude from the above discussion
that 
\begin{equation}
\label{eq-split}
\begin{split}
\frac{d}{d\tau}\int_{\mathbb{R}^N} q \, \eta_{R,\epsilon} \, \bar{\C}^{p_0}\, dy &\leq   \int_{\mathbb{R}^N} \big [\Delta \eta_{R,\epsilon} \, \bar{\C}^{p_0} a  + 2a  \, \nabla\eta_{R,\epsilon}\cdot\nabla\bar{\C}^{p_0} - \beta y\cdot\nabla\eta_{R,\epsilon} \, \bar{\C}^{p_0} \big ]\, q\, dy  \\
&+ \int_{\mathbb{R}^N} \big [a\, \Delta \bar{\C}^{p_0} - \beta\, y \cdot \nabla\bar{\C}^{p_0}+ (\alpha - N \beta)\, \bar{\C}^{p_0} \big ]\, q\, \eta_{R,\epsilon} \, dy\\ 
&= \int_{|x| \le \epsilon}   \big [\Delta \eta_{R,\epsilon}\,  \bar{\C}^{p_0} a  + 2a  \, \nabla\eta_{R,\epsilon}\cdot\nabla\bar{\C}^{p_0} - \beta y\cdot\nabla\eta_{R,\epsilon} \, \bar{\C}^{p_0} \big ]\, q\, dy \\
&+  \int_{R/2 \le |x| \le R}\big  [\Delta \eta_{R,\epsilon} \, \bar{\C}^{p_0} a  + 2a  \, \nabla\eta_{R,\epsilon}\cdot\nabla\bar{\C}^{p_0} - \beta y\cdot\nabla\eta_{R,\epsilon} \, \bar{\C}^{p_0} \big ]\, q\, dy\\
&+ \int_{\mathbb{R}^N} \big [a\, \Delta \bar{\C}^{p_0} - \beta\, y \cdot \nabla\bar{\C}^{p_0}  + (\alpha - N \beta)\, \bar{\C}^{p_0} \big ]\, q\, \eta_{R,\epsilon} \, dy .
\end{split}
\end{equation} 
Observe  that since $p_0 < 2m$
\begin{equation}
\begin{split}
\int_{|x| \le \epsilon} \big [a\, \Delta \eta_{R,\epsilon} \bar{\C}^{p_0}  &+ 2a \, \nabla\eta_{R,\epsilon}\cdot\nabla\bar{\C}^{p_0} - \beta y \cdot\nabla\eta_{R,\epsilon} \bar{\C}^{p_0} \big ]\, q \, dy  \\
&\leq  
C\, \epsilon^{- ( \frac{n+2}{2}\, p_0 + 2)}\,  |B_{\epsilon}(0)|  \le C\,\epsilon^{n - (p_0\,\frac{n+2}{2} + 2)} \to 0  \,\,\, \mbox{as} \,\,\, \epsilon\to 0.
\end{split}
\end{equation}
If we let $\epsilon\to 0$ in \eqref{eq-split}, then since the $\lim_{\epsilon\to 0} \eta_{R,\epsilon} = \eta_R$, where $\eta_R$ is a  smooth function with compact support in $R/2 \le |y| \le R$, we obtain 
\begin{equation}
\label{eq-use}
\begin{split}
\frac{d}{d\tau}\int_{\mathbb{R}^N} q \, \eta_{R} \, \bar{\C}^{p_0}\, dy & \leq    \int_{R/2 \le |x| \le R} \big [\Delta \eta_R \bar{\C}^{p_0} a  + 2a  \, \nabla\eta_R\cdot\nabla\bar{\C}^{p_0} - \beta y \cdot\nabla\eta_R \bar{\C}^{p_0} \big ]\, q\, dy \\&
+ \int_{\mathbb{R}^N} \big [a\, \Delta \bar{\C}^{p_0} - \beta\,y \cdot \nabla\bar{\C}^{p_0}  + (\alpha - N \beta)\, \bar{\C}^{p_0} \big ]\, q\, \eta_{R} \, dy.
\end{split}
\end{equation}

We claim next that if $\bar u_0 - \bar v_0 \in L^1(\bar{\C}^{p_0},\mathbb{R}^N)$ then $\bar u(\cdot,s) - \bar v(\cdot,s) \in L^1(\bar{\C}^{p_0},\mathbb{R}^N)$ for  $s \in [0,\tau]$ uniformly in $s$. Indeed recalling  that we have chosen   $p_0= p_2$  and  integrating the previous estimate over $[0, s]$
(with $s \in  [0,\tau]$) also using \eqref{eqn-strict}  we  get 
\begin{equation}
\begin{split}
&  \int_{\mathbb{R}^N} q(x,s)\, \eta_R\, dy \, - \, \int_{\mathbb{R}^N} q(x,0)\, \eta_R\, dy \\
&\leq  \int_{0}^{s}\int_{{R}/{2} \le |y| \le R} \,\big [\Delta \eta_R\, \bar{\C}^{p_0} a  + 2a\, \nabla\eta_R\cdot\nabla\bar{\C}^{p_0} - \beta y\cdot\nabla\eta_R \bar{\C}^{p_0} \big ]\, q\, dy \, d \bar s\\
&\le C(\tau)\, \int_0^{\tau}\int_{{R}/{2} \le |y| \le R} |y|^{- \big (\frac{2(1+p_0)}{1-m}+\gamma \big )}\, dy \, d\bar s  \le C(\tau).
\end{split}
\end{equation}
From the uniform integrability of  $q(\cdot,s)\in L^1(\bar{\C}^{p_0}, \mathbb{R}^N)$ on $s \in [0,\tau]$, 
we conclude that for any fixed $\tau >0$, we have 
\begin{equation}\label{eqn-RR}
\lim_{R \to \infty} \int_{0}^{\tau}\int_{R/2 \le |y| \le R} q(x,s) \, \bar{\C}^{p_0}\, dy\, ds=0.
\end{equation}
On the other hand, since  $|\Delta\eta_R| \le {C}/{R^2}$, $|\nabla\eta_R|  \le {C}/{R}$, $a(y,\tau) \le C\, |y|^2$ and $|\nabla \bar{C}^{p_0}| \le C\, |y|\,  \bar{C}^{p_0}$, \eqref{eq-use} implies
$$\frac{d}{d\tau}\int_{\mathbb{R}^N} q \, \eta_R\,  \bar{C}^{p_0}\, dy \le 
 \int_{\mathbb{R}^N} \big [a\, \Delta \bar{\C}^{p_0} - \beta\,y \cdot \nabla\bar{\C}^{p_0} + (\alpha - N \beta)\, \bar{\C}^{p_0} \big ]\, q\, \eta_{R} \, dy + C\int_{{R}/{2} \le |y| \le R} q\, \bar{C}^{p_0} dy. 
$$
Integrating the above differential inequality over $[0, \tau]$ and  letting $R \to \infty$,   while using 
\eqref{eqn-RR} and \eqref{eqn-strict},   gives
$$\int_{\mathbb{R}^N} |\bar u - \bar v|   \,  \bar{C}^{p_0}\, dy < 
\int_{\mathbb{R}^N} |\bar u_0 - \bar v_0|  \,  \bar{C}^{p_0}\, dy,$$
finishing the proof of the lemma. 
\end{proof}

\begin{proof}[Proof of Theorem \ref{thm-conv2}]
Once we have Lemma \ref{lem-contr-weight}, which is the analogue of Lemma 4.1 in \cite{DS}  we finish the proof of Theorem \ref{thm-conv2} in  the same way as the proof of Theorem 1.2 in \cite{DS}.
\end{proof}

\section{Solutions that live longer} \label{sec-longer}

In Proposition \ref{prop-tail} we showed how the  cylindrical tail shrinks  in solutions that start as being asymptotic to a cylinder at infinity.  In Theorems \ref{thm-conv} and  \ref{thm-conv2} we dealt with the extinction profile of the class of solutions that become extinct at the time that their cylindrical tail disappears. In this section we give the proof of Theorem \ref{thm-longer} that describes the precise extinction profile of a class of solutions that live longer than their  cylindrical tail.

\begin{proof}[Proof of Theorem \ref{thm-longer}] We claim that given our initial conditions there exists a $B_1 >0$ such that
$$u_0 \leq V_{\beta,B_1}(\cdot,0) \qquad \mbox{on} \,\, \R^N.$$
To show the claim first  note that by our assumption \eqref{eqn-u0} there exist  $B_1 > 0$ and $r_0 \geq 1$ sufficiently large  so that
$$u_0(x) \le V_{\beta,B_1}(x,0) \qquad \mbox{for} \,\,\, |x| \ge r_0.$$
Recall that $V_{\beta,B_1}(x,t) = (T-t)^{\alpha} g_{B_1}(x(T-t)^{\beta})$ for $(x,t) \in \mathbb{R}^N\times (-\infty,T)$ and the behavior of $g_{B_1}$ has been discussed in Theorem \ref{thm-ss-sing}.
On the other hand recall that in section \ref{sec-scal} we  defined $g_{B_{\lambda}}(y) := \lambda^{2/n} g_{B_0}(\lambda y)$ where $B_0$ is chosen so that the $\lim_{y \to 0} |y|^{{\alpha}/{\beta}} g_{B_0}(y) = 1$. Recall that $\alpha =  (2\beta+1)/n$. 
We also have $B_{\lambda} = B_0 \lambda^{-\gamma}$. Using that $\lim_{\lambda\to 0} (\lambda |y|)^{{\alpha}/{\beta}} g_{B_0}(\lambda y) = 1$ we have 
$$\lim_{\lambda\to 0} g_{B_{\lambda}}(y) =  |y|^{-{\alpha}/{\beta}} \lim_{\lambda\to 0}  \lambda^{2/n-\alpha/\beta} = 
|y|^{-{\alpha}/{\beta}}  \lim_{\lambda\to 0} \lambda^{-{1}/(n\beta)} = +\infty.$$
This convergence is uniform in $y$ on the set $\{|y| \le r_0\}$. This means that by choosing $\lambda$ sufficiently small (which corresponds to $B_{\lambda}$ sufficiently large)  and increasing  the previously chosen $B_1$, if necessary, so that $B_1 \ge B_{\lambda}$  we have 
$$u_0(x) \le V_{\beta,B_1}(x) \qquad \mbox{for} \,\,\, |x| \le r_0.$$
This concludes the proof of the claim.

Let $W_{\beta,K_1}$ be the corresponding forward solution defined by  \eqref{eqn-expan} so that
$$\lim_{t \to T^-} V_{\beta,B_1}(x,t) = \lim_{t \to T^+}  W_{\beta,K_1} (x,t) = K_1 \, |x|^{-\alpha/\beta}.$$
Notice  that the above convergence is uniform on compact subsets
of $\R^N \setminus \{ 0 \}$ and also in  $L^1_{loc}(\R^N)$. 
Since $u_0 \leq V_{\beta,B_1}(\cdot,0)$, it follows by the comparison principle that 
\begin{equation}
\label{eq-left-bound1}
u (x,t) \leq  V_{\beta,B_1}(x,t) \qquad \mbox{for} \,\,\, (x,t) \in \mathbb{R}^N \times [0,T].
\end{equation}
In particular,
$u(x,T) \leq  \lim_{t \to T^-} V_{\beta,B_1}(x,t) =K_1 \, |x|^{-\alpha/\beta}$, so again by the comparison principle we have
\begin{equation}
\label{eq-right-bound2}
u(x,t) \le W_{\beta,K_1}(x,t) \qquad \mbox{for}  \,\,\, x\in \mathbb{R}^N, \,\,\, t>T.
\end{equation} 
Note that since  $u(\cdot,t), V_{\beta,B}(\cdot,t) \in L^1_{loc}(\mathbb{R}^N)$ for all $t\in [0,T)$, we can apply Corollary \ref{cor-resc} to 
show that the rescaled solution $\bar u$ as in  \eqref{eq-resc-left} satisfies 
$$\int_{\mathbb{R}^N} |\bar{u}(y,\tau) - g_{\beta,B}(y)|\, dy \le e^{-(\beta N - \alpha)\, \tau} \int_{\mathbb{R}^N} |\bar{u}_0(y) - g_{\beta,B}(y)|\, dy.$$
Our  assumption  $u_0  - V_{\beta,B}(\cdot,0) \in L^1(\mathbb{R}^N)$  implies that 
$\bar{u}_0 - g_{\beta,B} \in L^1(\mathbb{R}^N)$  and  by \eqref{eq-left-bound1}  we have 
$$\bar{u}(y,\tau) \le g_{\beta,B_1}(y),   \qquad \mbox{for} \,\,\, (y,\tau) \in \mathbb{R}^N \times [0,T].$$
The same arguments as in the proof of Theorem \ref{thm-conv} imply that $\bar{u}(\cdot,\tau)$ converges  as $\tau\to\infty$ uniformly on compact sets of $\R^n \setminus \{0\}$, and also  in the $L^1$ norm, to the singular self similar profile $g_{\beta,B}(x)$. 
Moreover, since $\beta N > \alpha$ the $L^1$ convergence  is exponential.
\smallskip

We recall that  $\lim_{t\to T-} V_{\beta,B}(x,t) = K \, |x|^{-\alpha/\beta}$ for some $K >0$. Let $W_{\beta,K}$ be the corresponding forward solution  as in \eqref{eqn-expan} so that the $\lim_{t\to T^+} W_{\beta,K} = K \, |x|^{\alpha/\beta}$. By the $L^1$ contraction property applied to $u(\cdot,t) - W_{\beta,K}(\cdot,t)$ we have for $t > T$,
\begin{equation}
\begin{split}
\int_{\mathbb{R}^N} |u(x,t) &- W_{\beta,K}(x,t)|\, dx \le  \lim_{t\to T^+} \int_{\mathbb{R}^N} |u(x,t) - W_{\beta,K}(x,t)|\, dx \\
&= \lim_{t\to T^-} \int_{\mathbb{R}^N} |u(x,t) - V_{\beta,B}(x,t)|\, dx \le \int_{\mathbb{R}^N} |u_0(x) - V_{\beta,B}(x,0)|\, dx \le C.
\end{split}
\end{equation}
Under the rescaling defined by  \eqref{eq-resc-right} the previous estimate becomes
\begin{equation}
\label{eq-int-100}
\int_{\mathbb{R}^N} |\hat{u}(y,\tau) - h_{\beta,K_1}(y)|\, dy \le C e^{(n\beta - \alpha)\,\tau}
\end{equation}
for $\tau\in (-\infty, \tau^*)$ with $\tau^*:= \log (T^*/T)$.  Moreover the bound \eqref{eq-right-bound2} becomes 
$$0 \le \hat{u}(y,\tau) \le h_{\beta,K_1}(y) \qquad \mbox{forall} \,\,\, (y,\tau) \in \mathbb{R}^N\times (-\infty, \tau^*).$$ 
This together with \eqref{eq-int-100} yields the convergence of the rescaled solutions $\hat{u}(\cdot,\tau)$ as $\tau\to -\infty$
 uniformly on compact sets  of $\mathbb{R}^N\backslash \{0\}$  and exponentially in the $L^1$ norm to the singular self-similar profile $h_{\beta,K_1}$. This in particular
shows that $u(\cdot,T) > 0$. Finally, since $g_{\beta,K_1}(y) = O(|y|^{-(N+2})$, as $|y| \to \infty$, 
the same must hold for $u(\cdot,t)$ for $t >T$ since $W_{\beta,K_1}$ dominates $u$.  
\end{proof} 

As a consequence of the proofs of the Theorems \ref{thm-conv}, \ref{thm-conv2} and \ref{thm-longer} we have the following Corollary that 
in particular gives other  examples of solutions that live longer than the extinction time of their cylindrical tail.

\begin{cor}
\label{cor-asymp}
Let   $u:\R^N\times [0,T)\to\R$ be a  solution to \eqref{eqn-u} with the initial data $u_0 \in L^\infty_{loc}(\R^N)$ satisfying
\begin{equation}
\label{eq-diff-in-cond}
u_0(x) =  \left(\frac{C^* T}{|x|^2}\right)^{1/n}\, (1 - B\, |x|^{-\gamma} + o(|x|^{-\gamma}) \qquad \mbox{as} \,\,\, |x|\to\infty
\end{equation}
where $B > 0$ and $\gamma := \gamma_1$. 
\begin{itemize}
\item[(i)]
If $\beta \ge \beta_1$ and $N \ge 3$ and  $u_0 - U_{\beta,B}(\cdot,0) \in L^1(\R^N)$ for some $B > 0$, then we have the same conclusion as in Theorem \ref{thm-conv} (just replace the $C^0$ convergence on compact subsets of $\R^N$ by the uniform $C^0$ convergence on compact subsets away from the origin).
 \item[(i)]
If  $\beta_0 < \beta < \beta_1$ and $N \ge 6$ and $u_0 - U_{\beta,B}(\cdot,0) \in L^1(\R^N, \C^{p_0})$ for some $B > 0$ and $p_0$ as in Theorem \ref{thm-conv2}, the same  conclusion as in  Theorem \ref{thm-conv2} holds (just replace the $C^0$ convergence on compact subsets of $\R^N$ by the uniform $C^0$ convergence on compact subsets away from the origin).
\end{itemize}
\end{cor}

\begin{proof}
By the same arguments as in the proof of Theorem \ref{thm-longer} there exists a $B_1 > 0$ so that $u_0 \le V_{\beta,B_1}(\cdot,0)$ on $\R^N$. By the comparison principle we have $ u(x,t) \le V_{\beta,B_1}(x,t)$ for all $(x,t) \in \R^N\times [0,T)$. If we apply the rescaling \eqref{eq-resc-left} to $u$ this bound reads as $\bar{u}(y,\tau) \le g_{\beta,B_1}(y)$ for all $(y,t) \in \R^N \times [0, \infty)$.
We can now apply the proofs of Theorems \ref{thm-conv} and \ref{thm-conv2} to get the convergence statements in  (i) and (ii) respectively. The only difference is that, since $g_{\beta,B_1}$ is singular at the origin, 
we only have uniform  convergence on compact subsets of $\R^N \backslash \{0\}$, i.e. away from the origin. \end{proof}

\begin{rem}
One easily checks that $\gamma := \gamma_1 < (N+2)/{2}$. Consider  initial data  of the form 
$$u_0(x) := U_{\beta,B}(x) + f(a\, x)$$
for some $a >0$, where $f \in L^\infty(\R^N)$ and $f(x) = o(|x|^{-(N+2})$ as $|x|\to\infty$. Then $\gamma < (N+2)/{2}$ implies that 
$$u_0(x) = \left(\frac{C^*T}{|x|^2}\right)^{2/n} \big (1 - B\, |x|^{-\gamma} + o(|x|^{-\gamma} \big )$$
where $B >0$.  As in \cite{DS2} choose $a$ sufficiently small  that the vanishing time of the  solution $u$  to \eqref{eqn-u}  starting at $u_0$
is  $T^* > T$ while   by Proposition \ref{prop-cyl}  the cylindrical tail of $u$ becomes extinct at $T$. 
 We still have $u_0(x) - U_{\beta,B}(x) \in L^1(\R^N)$ so Corollary \ref{cor-asymp} applies to our solution.  
The same proof of Theorem 1.4 in \cite{DS2} yields that for $t >T$ 
$$u(x,t) \le {c(t)} \, {|x|^{-(N+2)}}  \qquad  \mbox{as}\,\, |x| \to +\infty.$$
\end{rem}

\end{document}